\documentclass[a4paper,twoside]{article}

\usepackage{color}
\definecolor{darkgreen}{cmyk}{1,0,1,.2}
\definecolor{m}{rgb}{1,0.1,1}
\definecolor{pink}{rgb}{1,0,1} 

\long\def\red#1{#1}

\long\def\green#1{#1}
\usepackage{amsmath,amsthm,amsfonts,latexsym,amscd,amssymb,enumerate}
\usepackage[pdftex]{hyperref}
\usepackage[all,arc,cmtip]{xy}

\swapnumbers
\theoremstyle{plain}
\newtheorem{theorem}{Theorem}[section]
\newtheorem{lemma}[theorem]{Lemma}
\newtheorem{corollary}[theorem]{Corollary}
\newtheorem{proposition}[theorem]{Proposition}

\theoremstyle{definition}
\newtheorem{definition}[theorem]{Definition}

\theoremstyle{remark}
\newtheorem{remark}[theorem]{Remark}

\newcommand{\PD}{\mathcal{PD}}
\newcommand{\db}{\mathcal{D}}

\newcommand{\reals}{\mathbb{R}}
\newcommand{\complexs}{\mathbb{C}}
\newcommand{\C}{\mathbb{C}}

\newcommand{\integers}{\mathbb{Z}}

\renewcommand{\j}{\jmath}
\DeclareMathOperator{\id}{id}
\newcommand{\boundary}[1]{\partial#1}

\newcommand{\abs}[1]{\left\lvert#1\right\rvert} %absolute value

\newcommand{\tensor}{\otimes}
\newcommand{\spg}{$G$-$\spin^c$ }
\newcommand{\into}{\hookrightarrow}
\newcommand{\onto}{\twoheadrightarrow}
\newcommand{\iso}{\cong}
\newcommand{\disjointunion}{\amalg}

\DeclareMathOperator{\End}{End}    %Endomorphisms
\DeclareMathOperator{\spin}{spin}
\DeclareMathOperator{\Spin}{Spin}

\DeclareMathOperator{\pr}{pr}
\DeclareMathOperator{\KK}{KK} 

\newcommand{\forget}[1]{}

\newcommand{\kk}[1]{[#1]}
\newcommand{\term}[1]{\emph{#1}}
\newcommand{\PinC}{\ensuremath{\Pin^c}}
\newcommand{\spinC}{\ensuremath{\spin^c}}
\newcommand{\SpinC}{\ensuremath{\Spin^c}}
\newcommand{\Sred}[1]{{S\hspace{-0.55em}/\hspace{0.15em}}_{#1}}
\newcommand{\Sreddual}[1]{{\overline{S\hspace{0.15em}}\hspace{-0.69em}/\hspace{0.05em}}_{#1}}
\newcommand{\Sreddualeven}[1]{{\overline{S\hspace{0.15em}}\hspace{-0.69em}/\hspace{0.05em}}_{#1}^{\hspace{0.25em}+}}
\newcommand{\Sfull}[1]{S_{#1}}
\DeclareMathOperator{\Pin}{Pin}
\DeclareMathOperator{\collapse}{collapse}
\DeclareMathOperator{\incl}{incl}
\DeclareMathOperator{\mul}{mul}
\DeclareMathOperator{\op}{op}
\DeclareMathOperator{\Cliff}{Cliff}
\newcommand{\CliffordC}{\complexs\ell}

{\catcode`@=11\global\let\c@equation=\c@theorem}

\allowdisplaybreaks[2]

\begin{document}
\pagestyle{myheadings}
\markboth{Paul Baum, Herv\protect{\'e} Oyono, Thomas Schick}{Equivariant geometric K-homology}

\date{}
%\date{Last compiled \today; last edited  \heuteIst or later}

\title{Equivariant geometric K-homology for compact Lie group actions}

\author{Paul
  Baum\thanks{\protect\href{mailto:baum@math.psu.edu}{email:~baum@math.psu.edu}; 
    \protect\href{http://www.math.psu.edu/baum}{www:~http://www.math.psu.edu/baum}\protect\\
Partially supported by the Polish Government grant N201 1770 33 and by an NSF research grant}\\ Pennsylvania State University\protect\\ State
  College\protect\\ USA \and Herv\'e
  Oyono-Oyono\thanks{\protect\href{mailto:oyono@math.cnrs.fr}{email:
      oyono@math.cnrs.fr}; \protect\href{http://math.univ-bpclermont.fr/~oyono/}{www:~http://math.univ-bpclermont.fr/~oyono/}}\\ Universit\'e Blaise 
  Pascal\\
  Clermont-Ferrand\\  
France \and Thomas Schick\thanks{
\protect\href{mailto:schick@uni-math.gwdg.de}{e-mail:
  schick@uni-math.gwdg.de};
\protect\href{http://www.uni-math.gwdg.de/schick}{www:~http://www.uni-math.gwdg.de/schick}\protect\\
partially funded by the Courant Research Center "`Higher order structures in Mathematics"'
within the German initiave of excellence
}\\
Georg-August-Universit\"at\protect\\ G{\"o}ttingen\\
Germany
\and
Michael Walter\thanks{\protect\href{mailto:michael.walter@gmail.com}{email: michael.walter@gmail.com}}\\
Georg-August-Universit\"at\protect\\ G{\"o}ttingen\\
Germany
}
\maketitle

\begin{abstract}
  Let $G$ be a compact Lie-group, $X$ a compact $G$-CW-complex. We define
 equivariant geometric K-homology groups $K^G_*(X)$, using an obvious
 equivariant  version of the $(M,E,f)$-picture of Baum-Douglas for
 K-homology.
 We define explicit natural transformations to and from equivariant K-homology
 defined  via 
 KK-theory (the ``official'' equivariant K-homology groups) and show that
 these are isomorphisms.
\end{abstract}

\section{Introduction}

K-homology is the homology theory dual to K-theory. For index theory, concrete
geometric realizations of K-homology are of relevance, as already pointed out
by Atiyah \cite{atiyah69}. In an abstract analytical setting, such a
definition has been given by Kasparov \cite{kasparov75}. About the same time,
Baum and Douglas \cite{baum-douglas82} proposed a very geometric picture of
K-homology (using manifolds, bordism, and so on), and defined a simple map to
analytic K-homology. This map was ``known'' to be an isomorphism. However, a
detailed proof of this was only published in \cite{BHS}.

The relevance of a geometric picture of K-homology extends to
equivariant situations. Kasparov's analytic definition of K-homology
immediately does allow for such a generalization, and this is considered to be
the ``correct'' definition. The paper \cite{BHS} is a spin-off of work on
a Baum-Douglas picture for $\Gamma$-equivariant K-homology, where
$\Gamma$ is a discrete group acting properly on a $\Gamma$-CW-complex. This
requires considerable effort because of the difficulty to find equivariant
vector bundles in this case. Emerson and Meyer give a very general geometric
description even of bivariant equivariant K-theory, provided enough such vector
bundles exist ---compare \cite{Emerson-Meyer}.

In the present paper, we give a definition of $G$-equivariant K-homology for
the case that $G$ is a compact Lie group, in terms of the ``obvious''
equivariant version of the $(M,E,f)$-picture of Baum and Douglas. Our main
result is that these groups 
indeed are canonically isomorphic to the standard analytic equivariant
K-homology groups. The main point of the construction is its simplicity, we
were therefore not interested in utmost generality.

In the case of a compact Lie group, equivariant vector
bundles are easy to come by, and therefore the work is much easier than in the
case of a discrete proper action. We will in part follow closely the work of
\cite{BHS}, and actually will omit detailed descriptions of the equivariant
generalizations where they are obvious. In other parts, however, we will
deviate from the route taken in \cite{BHS} and indeed give simpler
constructions. Much of our theory is an equivariant (and more geometric)
version of a general theory of Jakob \cite{MR1624352}. These constructions
have no generalization to proper actions of
discrete groups and were therefore not used in \cite{BHS}. Moreover, we will
use the full force of Kasparov's KK-theory in some of our analytic
arguments. The diligent reader is then asked to supply full arguments where
necessary.

%% \textbf{The following is a description of our work-schedule and will not
%%   appear in the final paper.}

%% Plan of the paper;  probably everything should be done for pairs.
%% \begin{enumerate}
%% \item Definition of $K^{G,geom}_*(X)$
%% \item Definition of the transformation $\alpha\colon KK^G_*(C(X),\complexs)\to
%%   K^{G,geom}_*(X)$. This uses Poincar\'e duality in equivariant KK-theory.

%%   We also use that we find a compact $G$-$\spin^c$ manifold $W$ with boundary which is
%%   $G$-homotopy equivalent to $X$  (for the moment: assume that this
%%   exists). Actually: it suffices to find one without boundary which retracts
%%   on $X$, which one gets e.g.~by doubling the first thing.
%% \item Definition of the transformation $\beta\colon K^{G,geom}_*(X)\to
%%   KK^G_*(C(X),\complexs)$. To show that the map is well defined, check
%%   compatibility with
%%   \begin{enumerate}
%%   \item disjoint union
%%   \item bordism
%%   \item vector bundle modification
%%   \end{enumerate}
%% \item Show that $\beta\circ\alpha=\id$
%% \item Show that $\alpha\circ \beta=\id$. For this:
%%   \begin{enumerate}
%%   \item Show that it suffices to consider the case $X=N$ is a compact
%%   $G$-$\spin^c$ manifold  ($N$ is the double of $W$).
%% \item show that each $(M,E,\phi)$ is equivalent to one with $\phi$ smooth
%% \item Show that each $(M,E,\phi)$ is equivalent to $(N,\phi!E,\id)$.
%%   \end{enumerate}

%% \end{enumerate}

\section{Equivariant geometric K-homology}

Let $G$ be a compact Lie group, $(X,Y)$ be a compact $G$-$CW$-pair with a 
$G$-homotopy retraction $(X,Y)\xrightarrow{j} (W,\boundary W)\xrightarrow{q}
(X,Y)$. We require that $(W,\boundary W)$ is  a smooth 
$G$-$\spin^c$ manifold with boundary. $G$-homotopy retraction means that $qj$
is $G$-homotopy equivalent to the identity (and the homotopy preserves $Y$). 

\begin{lemma}
  Every finite $G$-CW-pair, more generally every compact $G$-ENR and in
  particular every smooth compact $G$-manifold
  (absolute   or relative  to its boundary) has the required property,
  i.e.~is such a homotopy retraction of a manifold with boundary.
\end{lemma}
\begin{proof}
  
  This is trivial for a $G$-$\spin^c$ manifold.

  The following argument is partly somewhat sketchy, we leave it to the reader
  to add the necessary details.

  In general, by \cite{MR0394550}, every finite $G$-CW-complex $X$ has a
  (closed) 
  $G$-embedding into a finite dimensional complex linear $G$-space (using
  \cite{MR0087037}) with an open $G$-invariant neighborhood $U$ with a
  $G$-equivariant 
  retraction $r\colon U\to X$ onto $X$. Even better, every 
  such $G$-embeddings admits such a neighborhood retraction, using
  \cite{MR0033830}. In other words, a finite
  $G$-CW-complex is a $G$-ANR. By 
  \cite{MR2350547}, the converse is true up to $G$-homotopy equivalence. 

   {A complex $G$-representation in particular has a $G$-invariant
   $\spin^c$-structure, and therefore so has $U$.}
  Choose
  a $G$-invariant metric on $U$, {e.g.~the metric induced by a
   $G$-invariant 
   Hermitean metric on the $G$-representation}. Let $f$ be the distance to
  $X$, a $G$-invariant map on $U$. Choose $r>0$ such that $f^{-1}([0,r])$ is
  compact. This is possible since $X$ is compact: choose $r$ smaller than the
  distance from $X$ to the complement of $U$. Choose a smooth $G$-invariant
  approximation $g$ to $f$, i.e.~$g$ has to be sufficiently close to $f$ in
   the 
  chosen metric. {To construct $g$, we can first choose a
   non-equivariant approximation 
    and then average it to make it $G$-invariant.}  Choose a regular value
  $0<r'<r$ such that $V:= f^{-1}((-\infty,r'])$ is a neighborhood of $X$ and
   is a 
  compact manifold (necessarily a $G$-manifold) with boundary. Its double $W$
   is a 
  $G$-manifold with inclusion $i\colon X\to W$ (into one of the two copies) and
  with retraction $W\to X$ obtained as the composition of the ``fold map'' and
  the retraction $r$ (restricted to $V$).

  This covers the absolute case.

  If $(X,Y)$ is a $G$-CW-pair, choose an embedding $j$ of $X$ into some linear
  $G$-space $E$ of real dimension $n$ (with $\spin^c$-structure), and a
  $G$-invariant distance
  function. The distance to $Y$ then gives a $G$-invariant function $h\colon
  X\to [0,\infty)$ with $h(x)=0$ if and only if $x\in Y$. Consider $X\cup_Y X$
  with the obvious $\integers/2$-action by exchanging the two copies of
  $X$, and $G$-action by using the given action on both halves. Extend $h$ to
  a $G\times \integers/2$-equivariant map to $\reals$ with 
  $\integers/2$-action given by multiplication with $-1$ (and with trivial
  $G$-action). Let $q\colon X\cup_Y X\to X$ be the folding map.  Taking the product
 of $j\circ q$ with $h\colon X\cup_YX \to \reals$ (with trivial $G$-action on
 $\reals$), we obtain a $G{\times  \integers/2}$-embedding of $X\times_YX$ into $E\times \reals$.

{Construct now the $G\times \integers/2$-neighborhoood retract $U^+$
 and the manifold $W^+$ for this
 embedding as above. By construction, there is a well defined
 $\reals$-coordinate $r$ for all points in these neighborhoods and
 also in $W^+$ (a priori only a continuous function).  
The subset $\{r=0\}$ consisting exactly of the $\integers/2$-fixed points. The
 $\integers/2$-action 
 on $W^+$ is smooth. For each $\integers/2$-fixed point $x\in U^+$, (being an
 open subset of $E\times \reals$ with $\integers/2$-action fixing $E$ and
 acting as $-1$ on $\reals$) $T_xU^+\iso\reals^n\oplus \reals_-$ as
 $\integers/2$-representation (where $\reals$ denotes the trivial
 $\integers/2$-representation and $\reals_-$ denotes the non-trivial
 $\integers/2$-representation). The same is then true for any 
 $\integers/2$-submanifold with boundary of codimension $0$, and also for a double of such a
 manifold, like $W^+$.  }

{Because of this special structure of the $\integers/2$-fixed points it
  follows that $W:=W^+/\integers/2$ obtains the structure of a $G$-manifold
  with 
  boundary, here homeomorphic to the subset $\{r\ge 0\}$ (as this is a
  fundamental domain for the action of $\integers/2$). The boundary of
  $W=W^+/\integers/2$ is exactly the (homeomorphic) image of the fixed point
  set $\{r=0\}$. The 
  $G\times\integers/2$-equivariant retraction of $W^+$ onto $X\cup_Y X$
  descends to a $G$-equivariant
  retraction of $W$ onto $X=X\cup_YX/\integers/2$; the
  $\integers/2$-equivariance of the
  retraction implies that $\boundary W$, the image of the fixed point set is
  mapped under this retraction to $Y$ (the image of the $\integers/2$-fixed
  point set of $X\cup_Y X$), so we really get a retraction of the pair
  $(W,\boundary W)$ onto $(X,Y)$.}

\end{proof}

\begin{definition}\label{def:cycles}
  A \emph{cycle for the geometric equivariant $K$-homology of $(X,Y)$} is a triple $(M,E,f)$,
  where 
  \begin{enumerate}
  \item $M$ is a compact \green{smooth $G$-$\spin^c$ manifold (possibly with
      boundary and components of different dimensions)}
  \item  $E$ is a
    $G$-equivariant Hermitean vector bundle on $M$
  \item  $f\colon M\to X$ is
    a continuous $G$-equivariant map such that $f(\boundary M)\subset Y$.
\end{enumerate}
  Here, a $G$-$\spin^c$-manifold is a $\spin^c$-manifold with a given
  $\spin^c$-structure ---given as in \cite[Section 4]{BHS} in terms of a
  complex spinor bundle for $TM$, now with a $G$-action lifted to  and
  compatible with all the structure.

  We define isomorphism of cycles $(M,E,f)$ in the obvious way, given by maps
  which preserve all the structure (in particular also the $G$-action).

  The set of isomorphism classes becomes a monoid under the evident operation
  of disjoint union of cycles, we write this as $+$. This addition is obviously
  commutative. 
\end{definition}

More details about $\spin^c$-structures can be found in \cite[Section
4]{BHS}. All statements there have obvious $G$-equivariant  generalizations.

\begin{definition} 
If $(M,E,f)$ is a $K$-cycle for $(X,Y)$, then its \emph{opposite} $-(M,E,f)$ is the
$K$-cycle $(-M, E, f)$, where $-M$ denotes the manifold $M$ equipped with
the opposite $\spin^c$-structure.
\end{definition}

\begin{definition}
\label{bordism-def}
A \emph{bordism} of  K-cycles for the pair $(X,Y)$ consists of the following
data: 
\begin{enumerate}[\rm (i)]
\item A smooth, compact $G$-manifold $L$, equipped with a
  $G$-$\spin^c$-structure. 
\item A smooth, Hermitian $G$-vector bundle $F$ over $L$.
\item A continuous $G$-map $\Phi \colon L\to X$.
\item A smooth map $G$-invariant map $f\colon \partial L \to \reals$ for which
  $\pm 1 $ are regular values, and for which $\Phi [ f ^{-1} [ -1,1]]\subseteq
  Y$. 
\end{enumerate}
The sets $M_{+} = f^{-1} [+1,+\infty)$ and $M_{-} = f^{-1} (-\infty, -1]$ are
manifolds with boundary, and we obtain two K-cycles $(M_{+},
F|_{M_{+}}, \Phi|_{M_{+}})$ and  
$(M_{-}, F|_{M_{-}}, \Phi|_{M_{-}})$ for the pair $(X,Y)$.  We say that the
first is bordant to the opposite of the second. {We follow
  here \cite[Definition 5.5]{BHS}, and as above the role of $f$ is to be able
  to talk of bordism of manifolds with boundary without having to introduce
  manifolds with corners.}
\end{definition}

\begin{definition}\label{def:modification}
Let $M$ be a $G$-$\spin^c$-manifold and let $W$ be a $G$-$\spin^c$-vector
bundle of even dimension
over $M$.  Denote by $\mathbf 1$ the trivial, rank-one real vector bundle
(with fiberwise trivial $G$-action).    The direct sum $W\oplus \mathbf 1$ is a
$G$-$\spin^c$-vector bundle, and the total space of this bundle is equipped
with a $G$-$\spin^c$ structure in the canonical way, as in \cite[Definition
5.6]{BHS}. %%  This is 
%% because its tangent bundle fits into an exact sequence 
%% $$
%% \xymatrix{
%% 0 \ar[r] & \pi^*[W\oplus \mathbf 1] \ar[r] & T(W\oplus \mathbf 1)\ar[r] & \pi ^* [ TM] \ar[r] & 0 , 
%% }
%% $$
%% where $\pi$ is the projection from $W\oplus \mathbf 1$ onto $M$, so that, upon choosing a splitting, (or equivalently, choosing a Riemannian metric on the manifold $W\oplus \mathbf 1$ which is compatible with the above sequence) we have a direct sum decomposition 
%% $$
%%  T(W\oplus \mathbf 1) \cong \pi^*[W\oplus \mathbf 1] \oplus  \pi ^* [ TM] .
%% $$
%% Different splittings result in concordant $\spin^c$-structures. 

{Let $Z$ be the unit sphere bundle of the bundle $ \mathbf 1 \oplus W$ with
bundle projection $\pi$. 
    Observe that an element of $Z$ has the form $(t,w)$ with $w\in W$, $t\in
    [-1,1]$ such that $t^2+\abs{w}^2=1$. The subset $\{t=0\}$ is canonically
    identified with the unit sphere bundle of $W$, $\{t\ge 0\}$ is called the
    ``northern hemisphere'', $\{t\le 0\}$ the ``southern hemisphere''. The map
    $s\colon M\to Z; m\mapsto (1,z(m))$ is called the \emph{north pole
      section}, where $z\colon M\to W$ is the zero section.
Since $Z$ is \green{contained in} the boundary of the disk bundle, we may equip it with a natural
$G$-$\spin^c$-structure by first restricting  the given
$G$-$\spin^c$-structure on 
the total space of $\mathbf 1\oplus W$  to the disk bundle, and then taking
the 
boundary of this $\spin^c$-structure to obtain a $\spin^c$-structure on the
sphere bundle. }

  We construct a bundle  $F$ over $Z$ via \emph{clutching}: if $S_W$ is the spinor bundle of
  $W$ (a bundle over $M$), then $F$ is obtained from $\pi^*S^*_{W,+}$
  over the northern hemisphere of $Z$ and $\pi^*S^*_{W,-}$ over the southern
  hemisphere of $Z$ by gluing along the intersection, the unit sphere bundle
  of $W$, using Clifford
  multiplication 
  with the respective vector of $W$. It follows from the discussion in Appendix \ref{sec:bott-thom}
  that this bundle is isomorphic to $S^*_{v,+}$, 
% The vertical tangent bundle of  $Z$  has a natural $G$-$\spin^c$-structure.
% Denote by $S_V$ the 
% corresponding reduced spinor bundle ---the usual bundle of spinors on an even
% dimensional $\spin^c$-manifold--- (a $G$-Hermitean vector bundle) and let
% $
% F = S_{V, +}^* 
% $.
% In other words, define $ F$ to be
the dual of the even-graded part of the $\integers/2$-graded spinor
bundle $S_v$. {The latter in turn is obtain from  the
  vertical ($G$-$\spin^c$-)tangent bundle of  the sphere bundle of $\mathbf
  1\oplus W$.} 
  The \emph{modification} of a $K$-cycle $(M,E,f)$ associated to 
the bundle $W$ is the
  $K$-cycle 
$(Z,  F\otimes \pi^*E , f\circ \pi)$.
\end{definition}

\begin{definition}\label{def:equiv_relation}
  We define an equivalence relation on the set of isomorphism classes of
  cycles of Definition \ref{def:cycles} as follows. It is generated by the
  following three elementary steps:
  \begin{enumerate}
  \item \textbf{direct sum is disjoint union}. Given $(M,E_1,f)$ and $(M,E_2,f)$,
    \begin{equation*}
(M,E_1,f)+ (M,E_2,f) \sim (M,E_1\oplus E_2,f).
\end{equation*}
\item \textbf{bordism}. If there is a bordism of K-cycles $(L,F,\Phi)$ as in
  Definition \ref{bordism-def} with boundary the two parts $(M_1,E_1,f_1)$ and
  $-(M_2,E_2,f_2)$, we set $(M_1,E_1,f_1)\sim (M_2,E_2,f_2)$.
\item \textbf{modification}. If $(Z,F\otimes \pi^*E,f\circ\pi)$ is the
  modification of a K-cycle $(M,E,f)$ associated to the $G$-$\spin^c$
  bundle $W$, then $(Z,F\otimes \pi^*E,f\circ\pi)\sim (M,E,f)$. 
  \end{enumerate}
\end{definition}

\begin{definition}
  For a pair $(X,Y)$ as above, we define the \emph{equivariant geometric
  K-homology} $K^{G,geom}_*(X,Y)$ as the set of isomorphism classes of cycles
  as in Definition \ref{def:cycles}, modulo the equivalence relation of
  Definition \ref{def:equiv_relation}.

  {Disjoint union of $K$-cycles provides a structure of}
  $\integers/2\integers$-graded {abelian} group, graded by the parity of
  the dimension of the 
  underlying manifold of a cycle. 
\end{definition}

\begin{lemma}\label{lem:K-bundles}
  Given a compact $G$-$\spin^c$-manifold $M$ with boundary, a $G$-map $f\colon
  (M,\boundary M)\to (X,Y)$ and a class $x\in K^0_G(M)$, we get a well defined
  element $[M,x,f]\in K^{G,geom}_*(X,Y)$ by representing $x=[E]-[F]$ with two
  $G$-vector bundles $E,F$ over $M$ and setting
  \begin{equation*}
    [M,x,f]:= [M,E,f]-[M,F,f] \in K^{G,geom}_*(X,Y).
  \end{equation*}

  In the opposite direction, we can assign to each triple $(M,E,f)$ a
  triple $(M,[E],f)$ with $[E]\in K^0_G(M)$ the K-theory class represented
  by $E$.
\end{lemma}
\begin{proof}
  We have to check that this construction is well defined, i.e.~we have to
  check 
  that $[E\oplus H]-[F\oplus H]$ gives the same geometric K-homology class,
  but this follows from the relation ``direct sum-disjoint union''.
\end{proof}

\begin{remark}\label{rem:K-cycles}
  Lemma \ref{lem:K-bundles} allows to use a geometric picture of equivariant
  K-homology (for a compact Lie group $G$) where the bundle $E$ is replaced by
  a K-theory class $x$; and all other definitions are translated accordingly.
\end{remark}

\begin{definition}
  $K^{G,geom}_*$ is a $\integers/2$-graded functor from pairs of $G$-spaces to
  abelian groups. Given $g\colon (X,Y)\to (X',Y')$, we define the transformation
  $$g_*\colon K^{G,geom}_*(X,Y)\to K^{G,geom}_*(X',Y');\;
  g_*[M,E,f]:=[M,E,g\circ f].$$
 An inspection of our equivalence relation
  shows that this is well defined, and it is obviously functorial.

  Moreover, we define a boundary homomorphism
  \begin{equation*}
    \partial\colon K^{G,geom}_*(X,Y)\to K^{G,geom}_{*-1}(Y,\emptyset);
    [M,E,f]\mapsto [\boundary M,E|_{\boundary M}, f|_{\boundary M}].
  \end{equation*}
  Again, we observe directly from the definitions that this is compatible with
  the equivalence relation, natural with respect to maps of $G$-pairs and
  a group homomorphism.
\end{definition}

Our main Theorem \ref{theo:main} shows that we have (for the subcategory of
compact $G$-pairs which are retracts of $G$-$\spin^c$ manifolds) explicit
natural isomorphisms to $K^{G,an}_*$. In particular, we observe that on this
category $K^{G,geom}_*$ with the above structure is a $G$-equivariant homology
theory. 

\section{Equivariant analytic K-homology}

For $G$ a compact group and $(X,Y)$ a compact $G$-$CW$-pair, analytic equivariant
K-homology and analytic equivariant K-theory are defined in terms of bivariant
KK-theory:
\begin{equation*}
  K^{G,an}_*(X,Y):= KK^G_*(C_0(X\setminus Y),\complexs);\qquad K_{G}^*(X,Y):=
  KK^G_*(\complexs,C_0(\green{X\setminus Y})). 
\end{equation*}
Of course, it is well known that $K_G^0(X,Y)$ is
naturally isomorphic to the Grothen\-dieck group of $G$-vector bundle pairs over
$X$ with a isomorphism over $Y$.
Moreover, most constructions in equivariant K-homology and K-theory can be
described in terms of the Kasparov product in KK-theory.

\subsection{Analytic Poincar\'e duality}
The key idea we employ to 
 describe the relation between geometric and analytic K-homology  is
 Poincar\'e duality in the setting of equivariant KK-theory
 developed by Kasparov \cite{kasparov88}. 
An orientation  for equivariant K-theory is  given by a
$G$-$\spin^c$-structure.  This Poincar{\'e} duality was in fact
originally stated by Kasparov 
for general   oriented manifolds  by using the
Clifford  algebra $C_\tau(M)$. But  for a manifold $M$ with
  a   $G$-$\spin^c$-structure, the Clifford  algebra $C_\tau(M)$ used in
  \cite{kasparov88} is $G$-Morita equivalent to $C_0(M)$, the Morita
  equivalence being implemented by the sections of the spinor bundle.
 We refer to \cite{kasparov88} (see also \cite{Segal}) for the 
definition of the representable equivariant K-theory group
$RK_G^*(X)$ of a locally compact $G$-space $X$.
 We only recall  here that the cycles are given by 
the  cycles $(E,\phi,T)$ for Kasparov's bivariant $K$-theory group $\KK^G_*(C_0(X),C_0(X))$ 
such that the representation  $\phi$ {of $C_0(X)$} on $E$
is the one of the   {$C_0(X)$-Hilbert structure}. By forgetting this
extra requirement, we get an obvious homomorphism
$\iota_X \colon RK_G^*(X)\to\KK^G_*(C_0(X),C_0(X))$. Hence it makes sense to take the
Kasparov product with elements in $ K_{G}^{*}(C_0(X))=\KK^G_*(C_0(X),\C)$
and this gives rise to a product
$$RK_G^*(X)\times  K_{G}^{*}(C_0(X)) \to K_{G}^{*}(C_0(X));\,(x,y)\mapsto \iota_X(x)\otimes y.$$

 Recall that  for any
  $G$-$\spin^c$-manifold $M$, there is a
  fundamental class $[M]\in K_{G}^{\dim(M)}(C_0(M))$ associated to
  the Dirac element of the $G$-$\spin^c$-structure on
  $M$. Moreover, if $N$ is an open $G$-invariant subset
  of $M$, then $[N]$ is the restriction of $[M]$ to $N$, i.e the image
  of $[M]$ under the morphism $K_{G}^{\dim(M)}(C_0(M))\to
  K_{G}^{\dim(M)}(C_0(N))$ induced by the inclusion
  $C_0(N)\hookrightarrow C_0(M)$. This is the obvious equivariant
  generalization of \cite[Theorem 3.5]{BHS}, compare also the discussion of
  \cite[Chapters 10,11]{MR1817560}
  
\begin{theorem}\label{theo:main}
  Given any $G$-$\spin^c$-manifold $M$, the Kasparov product
  with the class $[M]$  gives an isomorphism
  \begin{equation*}
    \PD_M\colon  RK_G^*(M)\xrightarrow{\iso} K^{G}_{\dim
      M-*}(C_0(M));\, x\mapsto \iota_M(x)\otimes [M].
  \end{equation*}
\end{theorem}
\begin{remark}\
\begin{enumerate}
\item For a compact space, the equivariant K-theory and the
  equivariant representable K-theory coincide. In particular, for a
  compact $G$-$\spin^c$-manifold $M$, the Poincar{\'e} duality can be
  stated as  an isomorphism
\begin{equation*}
    \PD_M\colon K_G^*(M)\xrightarrow{\iso} K^{G,an}_{\dim M-*}(M),
  \end{equation*}
and moreover, for any complex $G$-vector $E$ on $M$, 
$\PD_M([E])$ is the class in $K^{G,an}_{\dim M-*}(M)= KK_{G}^{\dim M-*}(C_0(M),\C)$ associated
to the Dirac operator $D_M^E$ on $M$ with coefficient in the complex
vector bundle $E$.
\item Recall that  representable equivariant K-theory  is a  functor
  which is 
   invariant with respect to
  $G$-homotopies.  In particular, if $M$ is a compact
  $G$-$\spin^c$-manifold with boundary $\partial M$, then $M$ is
  $G$-homotopy equivalent to its interior $M\setminus\partial M$ and thus we get
  a natural identification $K_G^*(M)\cong RK_G^*(M\setminus\partial
  M)$ given by restriction to $M\setminus\partial M$ of the $C(M)$-structure.
In view of this, the  Poincar{\'e} duality for the pair $(M,\partial
M)$ can be stated in the following way 
\begin{equation*}
    \PD_M\colon K_G^{*} (M) \xrightarrow{\iso} K^{G,an}_{\dim(M)-*}(M,\partial M),
  \end{equation*}
  For a compact $G$-space $X$ and a closed $G$-invariant subset $Y$ of $X$,
  let us denote by $\iota_{X,Y}$, the composition $$K_G^*(X)\cong RK_G^*(X)\to
  RK_G^*(X\setminus Y)\stackrel{\iota_{X,Y}}{\to}KK^G_*(C_0(X\setminus
  Y),C_0(X\setminus Y)),$$ where the
  first map is induced by the inclusion $X\setminus Y\hookrightarrow X$. Then, with
  these  notations and under the identification $K_G^*(M)\cong RK_G^*(M\setminus\partial
  M)$, we get  for any $x$ in  $K_G^*(M)$ that
  $\PD_M(x)=\iota_{M,\partial M}(x)\otimes [M\setminus\partial M]$.
\end{enumerate}
\end{remark}

\begin{definition}\label{def:of_trafos}
  
  We are now in the situation to define the natural isomorphisms
\begin{equation*}
  \begin{split}
    \alpha &\colon K_*^{G,geom}(X,Y)\to K_*^{G,an}(X, Y)\\
    \beta &\colon K_*^{G,an}(X,Y)\to K_*^{G,geom}(X,Y).
\end{split}
\end{equation*}

To define $\alpha$, let $(M,E,f)$ be a cycle for geometric K-homology, with
$E$ a complex $G$-vector bundle on $M$. Then we set
\begin{equation*}
  \alpha([M,E,f]):= f_*(\PD_M([E])).
\end{equation*}
To define $\beta$, given $x\in K_k^{G,an}(X, Y)$, choose a
retraction $ (X,Y)\xrightarrow{j} (M,\boundary M)\xrightarrow{p} (X,Y)$ with
$M$ a compact $G$-$\spin^c$ manifold with $\dim(M)\equiv k\pmod 2$ (such a
manifold exists by assumption, if the parity is not correct just take the
product with $S^1$ with trivial $G$-action). Then set
\begin{equation*}
  \beta(x):= [ M, \PD_M^{-1}(j_*(x)), p].
\end{equation*}

\end{definition}

\begin{lemma}\label{lem:alpha_is_homom}
  The transformation $\alpha$ is compatible with the relation ``direct
  sum---disjoint union'' of the definition of $K^{G,geom}_*(X,Y)$. Under the
  assumption that $\alpha$ is well defined, it is a 
  homomorphism.
\end{lemma}
\begin{proof}
  \begin{multline*}
  \alpha([M,E,f]+[N,F,g])= \alpha([M\disjointunion N, E\disjointunion F,
  f\disjointunion g])\\
  = (f\disjointunion g)_*(\PD_M([E])\oplus
  \PD_N([F]))=f_*(\PD_M([E]))+g_*(\PD_N([F])). 
\end{multline*}
This implies both assertions, as $\PD$ and $f_*$ are both homomorphisms.
\end{proof}

To prove that both maps are well defined and indeed inverse to each other  we
need a few more properties of Poincar\'e duality which we collect in the
sequel. These statements are certainly well known, for the convenience of the
reader we give proofs of most of them in an appendix.

We first relate Poincar\'e duality to the Gysin homomorphism, and also
describe vector bundle modification in terms of the Gysin homomorphism. 

Let $f\colon M\to N$ be a smooth $G$-map between two
compact $G$-$\spin^c$-manifolds without boundary. We use, as a special
case of 
\cite[Section 4.3]{KasparovSkandalis} (see also {\cite[Section 7.2]{MR1671260}}), the Gysin
element $f!$ in
$KK^G_{\dim M-\dim N}(C(M),C(N))$. It has the functoriality property that if
$f\colon M\to N$ and $g\colon N\to 
N'$ are two smooth $G$-maps  between 
compact $G$-$\spin^c$-manifolds, then $f!\otimes g!=(g\circ f)!$.
We will also need the corresponding construction for manifolds with
boundary. We recall all this in the appendix.

\begin{lemma}
  \label{lem:K-mod}
  An equivalent description of vector bundle modification, using Remark \ref{rem:K-cycles}, is given as follows:
  
  Let $(M,x,\phi)$ be a cycle for $K^{G,geom}_*(X,Y)$, with $x \in K^0_G(M)$, and let $W$ be a $G$-$\spinC$-vector bundle over $M$ of even rank.
  Let $\pi\colon Z \rightarrow M$ be the underlying $G$-manifold of the modification with respect to $W$.
  Recall from definition \ref{def:modification} that $s\colon M \rightarrow Z$ is the north pole section and that the bundle $F$ is obtained via clutching.
  Then the vector bundle modification $(Z, \pi^*(x) \otimes [F], \phi \circ \pi)$ of $(M,x,\phi)$ along $W$ is bordant
  to the cycle $(Z, s!(x), \phi \circ \pi)$ (where we interprete $s!(x)$ as
  an element of $K^0_G(Z)$ using \green{formal Clifford} periodicity).
\end{lemma}
\begin{proof}
%  \pink{This is a new, shorter proof which uses the results of the appendix (please delete this comment after checking).}
  The $G$-vector bundle $F^\infty$ from Proposition \ref{prop:thom-topo}, the topological description of the Thom isomorphism, is pulled back from $M$,
  hence $\pi^*(x) \otimes [F^\infty]$ extends to the disk bundle of $W \oplus {\bf 1}$ and we conclude that the first cycle is bordant to
  \begin{equation*}
    (Z, \pi^*(x) \otimes ([F] - [F^\infty]), \phi \circ \pi).
  \end{equation*}
  We will now show that its $K$-theory class $\pi^*(x)\otimes([F]-[F^\infty])$
  agrees with 
$$s!(x) = x \otimes_{C(M)} b_W \otimes_{C_0(W)} \kk{\theta_M},$$
 where $\theta_M$ is the inclusion
  $C_0(W) \subseteq C(Z)$ of $C^*$-algebras \green{and $b_W$ is the ``Bott element'' (compare Remark \ref{rem:other_defs})}.  
  Indeed, using the $KK$-picture of the tensor product of vector bundles and
  commutativity of the exterior Kasparov product we find that, as element of
  $KK^G(\complexs,C(Z))$, the tensor product of vector bundles
 \green{ $\pi^*(x)\tensor([F]-[F^\infty])$ is given by
  \begin{equation*}
%    \begin{split}
    %&= 
x \otimes_{C(M)}(\kk{1_{C(M)}}\otimes ([F] - [F^\infty])) \otimes_{C(M \times Z)} \kk{\mu'}
%    \end{split}
  \end{equation*}
  where $\mu\colon C(Z \times Z) \rightarrow C(Z)$ and $\mu':=\mu\circ
  (\pi\times \id_Z)\colon C(M \times Z) \rightarrow C(Z)$} are pointwise multiplication.
  The claim now follows from comparing the right-hand Kasparov product (which can be computed explicitely) with $b_W \otimes_{C_0(W)} \kk{\theta_M}$
  (which is straightforward using the topological description in Proposition
  \ref{prop:thom-topo} \green{and the fact that off the tubular neighborhood
    $W$, the two bundles are isomorphic}).
\end{proof}

{From now on, we will use the following notation: if $f\colon M\to N$
is  a $G$-map between $G$-$\spin^c$-manifolds and $E$ is a complex vector bundle
over $M$, then $f!E$ will stand for the element $f![E]$ of $K^{*+n-M}_G(N)$.}
It is well known that the Gysin
map  and functoriality in K-homology are intertwined by Poincar{\'e}
duality. This is the key for proving that $\alpha$ is compatible with vector
bundle modification, using the description of the latter given in Lemma
\ref{lem:K-mod}. We will prove the next assertion in \ref{sec:Gysin_and_PD}.
\begin{lemma}\label{lem-gysin}
Let $f\colon M\to N$ be a $G$-map between $G$-$\spin^c$ manifolds with
$m=\dim M$ and $n=\dim N$, possibly with boundary. Assume that $f(\boundary
M)\subset\boundary N$. Then
we have the following commutative diagram
\begin{equation*}
    \begin{CD}
      K_G^*(M) @>{\PD_M}>> K^{G,an}_{m-*}(M,\boundary M)\\
       @Vf!VV @VVf_*V\\
      K_G^{*+n-m}(N) @>{\PD_N}>> K^{G,an}_{m-*}(N,\boundary N).
   \end{CD}
  \end{equation*}
\end{lemma}

%% \begin{remark}
%%   Note that the description of vector bundle modification of Lemma
%%   \ref{lem:K-mod} will be quite convenient to work with, but is less
%%   explicit and geometric than the original definition.
%% \end{remark}

\begin{lemma}\label{lem:alpha_and_bundle_modification}
  The transformation $\alpha$ of Definition \ref{def:of_trafos} is compatible
  with vector bundle modification.
\end{lemma}
\begin{proof}
  The assertion is a direct consequence of Lemma \ref{lem:K-mod} and Lemma
  \ref{lem-gysin}. Explicitly, if $(M,E,f)$ is a cycle for
  $KK^{G,geom}_*(X,Y)$ and $(Z,s!(E),f\circ \pi)$ the result of vector bundle
  modification according to Lemma \ref{lem:K-mod},  then
  \begin{equation*}
    \begin{split}
      \alpha(Z,s!(E),f\circ \pi) &= f_*\pi_* \PD_Z(s!(E))\\
       &\stackrel{\text{Lemma \ref{lem-gysin}}}{=} f_*\pi_* s_*\PD_M(E)
         \stackrel{\pi\circ s=\id}{=} f_*\PD_M(E) \\
         &= \alpha(M,E,f).
    \end{split}
  \end{equation*}
\end{proof}

We now recall that, in the usual long exact sequences in K-homology, the
boundary of the fundamental class is the fundamental class, or, formulated
more casually: the {boundary of the Dirac element  is the Dirac
  element  of the boundary}. To deal with bordisms of
manifolds with boundary, we actually need a slightly
more general version as follows, which we prove in Appendix \ref{sec:appendixprooflemma}.
\begin{lemma}\label{lem-boundary}
  Let $L$ be a \spg manifold with boundary $\partial L$, let $M$ be a
  $G$-invariant submanifold of $\partial L$ with boundary $\partial M$ such that
  $\dim M=\dim L-1$ and let   $\partial\in
KK^G_1(C_0( M\setminus\partial M),C_0(L\setminus\partial L))$ be the boundary element
associated to the exact sequence
$$0\to C_0(L\setminus\partial L)\to C_0( (L\setminus \partial L)
 \cup (M\setminus\partial M))\to C_0(M\setminus \partial M)\to 0.$$
Then $[\partial]\otimes [L\setminus \partial L]=[M\setminus \partial M]$.
\end{lemma}

\begin{corollary}\label{cor-PDLM}
With notation of Lemma \ref{lem-boundary}, the following diagram
commutes
  \begin{equation*}
     \begin{CD}
      K_G^n(L) @>>> K_{G}^{n}(M)\\
      @VV{\PD_L}V @VV{\PD_{ M}}V\\
      K^{G,an}_{\dim L-n}(L,\partial L) @>(-1)^n\cdot\partial\otimes>> K^{G,an}_{\dim
        L-n-1}(M,\partial M),
    \end{CD}
   \end{equation*}
where the top  arrow is induced by the inclusion $i\colon M\into L$.
\end{corollary}
\begin{proof}
Fix $x\in K_G^*(L)$ 
and denote by $x|_{M}$ {the image of   $x$ under the homomorphism
$K_G^*(L)\to K_G^*(M)$ induced by the inclusion $M\hookrightarrow L$}. Then we get
\begin{eqnarray*}
\partial\otimes \PD_L(x)&=&\partial\otimes\iota_{L,\partial
  L}(x)\otimes [L,\partial L]\\
&=&{(-1)^{\deg x}}\iota_{M, \partial M}(x|_{ M})\otimes\partial\otimes [L\setminus
\partial L]\\
&=&\iota_{M, \partial M}(x|_{ M})\otimes[M\setminus \partial M]\\
&=&\PD_{ M}(x|_{M}),
\end{eqnarray*}
where the second equality is  a well known consequence of the naturality
of 
boundaries and is proved in Lemma \ref{lem:boundary_and_PD_vorbereteit}
and where the third equality holds by Lemma \ref{lem-boundary}.
\end{proof}

%% As a special case of the last corollary, we also get
%% \begin{corollary}
%% With notation of Corollary  \ref{cor-boundary}, the following diagram
%% commutes
%%   \begin{equation*}
%%      \begin{CD}
%%       K_G^0(M) @>>> K_{G}^{0}(\boundary M)\\
%%       @VV{\PD_M}V @VV{\PD_{\boundary M}}V\\
%%       K^{G,an}_{\dim M}(M) @>[\partial]\otimes >> K^{G,an}_{\dim M-1}(\boundary M),
%%     \end{CD}
%%    \end{equation*}
%% where the top  arrow is induced by restriction from $M$ to
%% $\partial M$.
%% \end{corollary}

\begin{lemma} The transformation $\alpha$ is compatible with the bordism
  relation of $K^{G,geom}_*(X,Y)$, i.e.~let $(L,F,\Phi,f)$ be a bordism for a
  $G$-$CW$-pair $(X,Y)$. Then, with
notations of Definition \ref{bordism-def},
\begin{multline*}
  \alpha(M^+,F|_{M^+},\Phi|_{M^+}) =
  (\Phi|_{M^+})_*\PD_{M^+}([F|_{M^+}])\\
  =-(\Phi|_{M^-})_*\PD_{M^-}([F|_{M^-}])=
  \alpha(M^-,F|_{M^-},\Phi|_{M^-}).
\end{multline*}
\end{lemma}
\begin{proof}
If we  set $M=M^-\amalg M^+$, this amounts to prove 
that $(\Phi|_{M})_*\PD([F|_{M}])=0$ in
$K^{G,an}_*(X,Y)=KK^G_*(C_0(X\setminus Y),\C)$.
But this a consequence of Corollary \ref{cor-PDLM}, together with
naturality of boundaries  in the following commutative diagram with
exact rows
\begin{equation*}
     \begin{CD}
 0 @>>>   0@>>> C_0(X\setminus Y)@>>>  C_0(X\setminus Y)@>>> 0\\
&&    @VVV   @VVV @VVV  \\
 0 @>>>  C_0(L\setminus \partial L)  @>>> C_0(L\setminus
 f^{-1}([-1,1]))@>>>  C_0(M\setminus \partial M)@>>> 0  
    \end{CD},
   \end{equation*}
where the middle and right vertical arrows are induced by $\Phi$.
\end{proof}

We are now in the situation to state and prove our main theorem.
\begin{theorem}
  The transformations $\alpha$ and $\beta$ of Definition \ref{def:of_trafos}
  are well defined and inverse to each other natural transformations for
  $G$-homology theories.
\end{theorem}
\begin{proof}
  Lemmas \ref{lem:alpha_is_homom}, \ref{lem:alpha_and_bundle_modification},
  and \ref{lem-boundary} together imply that $\alpha$ is a well defined
  homomorphism. If we fix, for given $(X,Y)$ the manifold 
  $(M,\boundary M)$ which 
  retracts to $(X,Y)$ (or rather two such manifods, one for each parity of
  dimensions), then $\beta$ also is well defined. As soon as we show that
  $\beta$ is inverse to $\alpha$ we can conclude that it does not depend on
  the choice of $(M,\boundary M)$.

  It is a direct consequence of the construction (and of naturality of
  K-homology) that $\alpha$ is natural with respect to maps $g\colon (X,Y)\to
  (X',Y')$.

Corollary \ref{cor-PDLM} implies
that $\alpha$ is compatible with the boundary maps of the long exact sequence
of a pair, and therefore a natural transformation of homology theories
(strictly speaking, we really know that $K^{G,an}_*$ is a homology theory only
after we know that $\alpha$ is an isomorphism).

  We now prove that $\alpha\circ\beta=\id$. Fix $x\in K^{G,an}_*(X,Y)$. Then
  \begin{equation*}
    \begin{split}
      \alpha(\beta(x)) &= \alpha([M,\PD^{-1}(j_*(x)),p])
      =p_*(\PD\circ\PD^{-1}(j_*(x)))\\
      &= p_*j_*(x)=x
    \end{split}
  \end{equation*}

  The proof of $\beta\circ\alpha=\id$ is given in the next section.
\end{proof}

%% \section{Analysis of the equivariant Dirac operator}

\section{Normalization of geometric cycles}

The goal of this section is to prove that $\beta\circ\alpha\colon
K^{G,geom}_*(X,Y)\to K^{G,geom}_*(X,Y)$ \green{is the identity for a compact
$G$-pair $(X,Y)$ and for any choice of retraction of $(X,Y)$ (which a priori enters the
definition of $\beta$)}.

 % is the identity whenever $(X,Y)$ is
% a compact $G$-pair with a retraction $(X,Y)\xrightarrow{j} (N,\boundary
% N)\xrightarrow{p} (X,Y)$.

We prove first the result for a pair
$(X,Y)=(N,\partial N)$, where 
$N$ is a compact $G$-$\spin^c$-manifold with boundary $\partial N$.

\green{We start with the construction of $\beta$ given by the choice
of the particular retraction $\id_N\colon N\to N$. We will show that with this choice
$\beta\alpha=\id$. This implies of course that $\alpha$ is
invertible. This in turn means that any left inverse is equal to this
inverse. As we already know that the a priori different versions of $\beta$,
depending a priori on different retractions of $(N,\boundary N)$, are all left
inverses of $\alpha$, they are
all equal, and equal to $\alpha^{-1}$.}

% Since $\alpha$ and $\beta$ are natural, the above
% map is a direct summand of $\beta\circ\alpha\colon K^{G,geom}_*(N,\boundary
% N)\to K^{G,geom}_*(N,\boundary N)$. It therefore suffices to show that
% $\beta\circ\alpha=\id$ for $G$-$\spin^c$ manifolds.

%% If $(M,E,f')$ is a cycle for $K^{G,geom}_*(N,\boundary N)$ then we can replace
%% $f'$ by a $G$-homotopic smooth map $f$. Putting the homotopy on a cylinder, we
%% observe that we can replace $(M,E,f)$ by an equivalent cycle $(M,E,f)$ with
%% $f\colon (M,\boundary M)\to (N,\boundary N)$ smooth. This will be assumed from
%% now on.

Fix now $(M,x,f)$ a cycle for $K_*^{geom}(N,\boundary N)$ as above,
with $x$ in $K^0_G(M)$. Then 
\begin{equation*}
  \beta(\alpha[M,x,f])=[M,\PD^{-1}f_*\PD(x),\id]\stackrel{\text{Lemma
  \ref{lem-gysin}}}{=} [M,f!x,\id].
\end{equation*}

\begin{theorem}\label{theo:bord}
  Let $h\colon (M,\boundary M)\into (N,\boundary N)$ be the inclusion of a
  \spg submanifold, $E$ a complex $G$-vector bundle on $M$ ---or more generally
  an element of $K_G^0(M)$---  and let $f\colon (N,\boundary N)\to
  (X,Y)$ be a
  $G$-equivariant continuous map, where $(X,Y)$ is a $G$-space. Let
  $\nu$ be the normal bundle of $h$. 
  Fix the trivial complex line bundle on $N$. Then the vector bundle
  modification of $(M,E,f\circ h)$ ``along'' $\complexs\oplus \nu$ (with its canonical
  spin${}^c$-structure) and of $(N,h!E,f)$  ``along'' $\complexs\times
  N$ are bordant.  In particular, 
  \begin{equation*}
[M,E,f\circ h]= [N,h!E,f] \in K^{G,geom}_*(X,Y).
\end{equation*}
\end{theorem}
\begin{proof}
  \green{In the situation at hand,} we just can write down the bordism between
  the two cycles. 
  Recall the construction of vector 
  bundle modification (of $N$ along $\complexs\times N$): we consider
  $\complexs\times \reals\times N$, equip this with the standard Riemannian
  metric, and consider the unit disc bundle $D^3\times N$ with its sphere bundle
  $S^2\times N$ within this bundle. It comes with a canonical ``north-pole
  inclusion'' $i\colon N\to S^2\times N$, and the modificaton is $(S^2\times
  N, i!h!E, f\circ\pr_N)$. \green{Observe that, if $N$ has a boundary, so has
  $S^2\times N$, and $D^3\times N$ is a manifold with corners.}

\green{  Fix $\epsilon>0$ small enough  and an embedding of $\nu$ into $N$ as tubular
  neighborhood of $M$. Fix a {$G$-invariant} Riemannian metric on $\nu$. Then the
  $\epsilon$-disk bundle and the $\epsilon$-sphere bundle of $\complexs\oplus 
  \reals\oplus\nu $ are contained   in $D^3\times N$, and if we remove the
  $\epsilon$-disk bundle we get a manifold $W$ with two parts of its boundary
  being 
  $S^2\times N$ and the sphere bundle of $\complexs\oplus\reals\oplus\nu$,
  i.e.~the underlying manifold $S$ of the modification of
  $(M,E,f\circ h)$, with its
  north pole embedding $i_M\colon M\to S$. In general, $W$ is a manifold with
  corners. But all these  corners
can be viewed $G$-equivariantly as the corners of an (open)
$G$-submanifold  $[0,1[\times Z$ of $W$, where $Z$ is a
$G$-$\spin^c$-manifold with boundary. 
Straightening the corners, there exist
 a $G$-$\spin^c$-manifold with boundary $W'$ and a $G$-equivariant
 homeomorphism 
$\upsilon\colon W'\stackrel{\cong}{\longrightarrow} W$
 with a smooth $G$-invariant map $\Phi\colon \boundary W'\to \reals$ regular
 at 
 $-1$ and at $1$ such that
 $\upsilon$ restricts to diffeomorphisms
  $$\Phi^{-1}([1,+\infty[)\stackrel{\cong}{\to} S^2\times N\text{ and
  }\Phi^{-1}(]\infty,-1])\stackrel{\cong}{\to} S;$$
 $\upsilon(\Phi^{-1}([-1,1]))\subset (D^3 \times \partial N)\cap
  W$%,
%and $M'_{\epsilon,1}:=\upsilon^{-1}(e(M\times[\epsilon,1]))$ is a
%  $G$-$\spin^c$-submanifold of  $W'$
. Then $W'$ is an underlying bordism
  between $S^2\times N$ and $S$ as in
  Definition \ref{bordism-def}.}

\green{  Observe that we have a canonical embedding $e\colon [\epsilon,1]\times M\into
  W\xrightarrow{\upsilon}W'$, using the $\reals$-coordinate of
  the vector bundle for the first map.}

  We actually get cartesian diagrams
  \begin{equation*}
    \begin{CD}
      M @>{i_1}>> [\epsilon,1]\times M \\
      @VV{i\circ h}V @VV{e}V\\
      S^2\times N @>j>>          W'
    \end{CD},\qquad
    \begin{CD}
      M @>{i_\epsilon}>> [\epsilon,1]\times M \\
      @VV{i_M}V @VVeV\\
      S @>{j'}>> W'
    \end{CD}.
  \end{equation*}

  Consider $e!(\pr_M^* E)$ on $W'$, with $\pr_M\colon [\epsilon,1]
  \times M \to
  M$ the obvious map.  We claim that $(W',e!(\pr_M^*E),f\circ
  \pr_N\circ\upsilon^{-1})$ is a bordism (in 
  the sense of Definition \ref{bordism-def}) between the two cycles we
  consider.

  Obviously, the boundary has the right shape, and
  $f\circ\pr_N\circ\upsilon^{-1}\colon W'\to X$ 
  restricts on $S^2\times N$ to the correct map. The restriction 
  % $f\circ\pr_N\circ\upsilon^{-1}|_S$ 
  to $S$ is homotopic to the map of the vector bundle modification
  of $(M,f\circ h,E)$ ---it is not equal, because one has to
 take the projection of the normal bundle of $M$ in $N$ onto $M$ into
 account. An easy modification of $f\circ \pr_N$ will produce a true
  bordism. Moreover, \red{$\upsilon(\Phi^{-1}(-1,1))\subset D^3\times
    \boundary N$} is mapped to $Y$ (and we can
  choose our modified $f\circ \pr_N$ such that this property is preserved).
%  \red{(indeed, we can replace $\pr$ by a $G$-equivariant
%     smooth map
%   $\pr'\colon W\to M$ such that $\pr=\pr'$ on $S^2\times N$ and $\pr'$  on
%   $S$  is the
%   map of the vector bundle modification)}.

%   The final point concerns the K-theory class. Here we use the compatibility of
%   pullback and push-down along cartesian products, {i.e.~that $j^*\circ
%   e!=(i\circ h)!\circ (i_1)^*$ and $(i_M)!\circ (i_\epsilon)^*=(j')^*e!$.
%   The restriction $j^*(e!(\pr_M^*E))$ of $e!(\pr_M^* E)$ to $S^2\times N$
%   therefore is
%   equal to 
% $$
% j^*(e!(\pr_M^*E))=(i\circ h)! (i_1^*\pr_M^*)E= i!h!E,$$
%  and
% the restriction of $e!(\pr^*_M E)$ to $S$ is equal to
%   $$(j')^*e!(\pr^*_ME)=({i_M})!i_\epsilon^*\pr_M^* E = {i_M}! E.$$}
  The claim is proved.
\end{proof}

We now finish the proof that $\beta(\alpha[M,E,f])=[N,f!E,\id]$
equals $[M,E,f]$. For
this, choose a finite dimensional $G$-representation $V$ and a $G$-embedding
$j_V\colon M\to V$ (this is possible because $G$ is a compact Lie
group and $M$
is compact, compare e.g.~\cite{MR0087037}). Observe
that $j_V$ is $G$-homotopic to the constant map with value
$0$. Embed $V$ into its one-point compactification $V^+$, a sphere (it can
also be realized as the unit sphere in $V\oplus \reals$). By composition we
obtain a $G$-embedding $j\colon M\to V^+$ which is still homotopic to the
constant map $c\colon M\to V^+$ with value $0$.

We obtain an embedding $M\xrightarrow{(f,j)} N\times V^+$, with $\pr_N\circ
(f,j)=f$.

By Theorem \ref{theo:bord} therefore
\begin{equation}
  \label{eq:first_inc}
  [M,E,f] = [N\times V^+,(f,j)!E, \pr_N].
\end{equation}
On the other hand, $(f,j)\colon M\to N\times V^+$ is $G$-homotopic to
$(f,c)\colon M\to N\times V^+$. Lemma \ref{lem-gysin} shows that $(f,j)!E$
depends only on the homotopy class of the map. Therefore
\begin{equation}
  \label{eq:homotopy}
  [N\times V^+,(f,j)!E,\pr_N]= [N\times V^+,(f,c)!E,\pr_N].
\end{equation}
Finally, $(f,c)=(\id_N,c)\circ f$, and $(\id_N,c)\colon N\to N\times V^+$ is
an embedding with $\pr_N\circ (\id_N,c)=\id_N$. Using functoriality of the
Gysin homomorphism and Theorem
\ref{theo:bord} again, we obtain
\begin{equation}
  \label{eq:second_inc}
  [N,f!E,\id] = [N\times V^+,(f,c)!E,\pr_N].
\end{equation}

This finishes the proof of our main theorem for a compact $G$-$\spin^c$-manifold
with boundary.
Now if   $(X,Y)$ is
a compact $G$-pair with a retraction $(X,Y)\xrightarrow{j} (N,\boundary
N)\xrightarrow{p} (X,Y)$, \green{let us consider 
\begin{equation}\label{commuting_square}
  \begin{CD}
     {K_{*}^{G,geom}(X, Y)} @>{j_*}>>  {K_{*}^{G,geom}(N,\partial N)}\\
     @VV{\beta\alpha}V   @VV{\beta_N\alpha_N=\id}V\\
{K_{*}^{G,geom}(X, Y)} @>{j_*}>>  {K_{*}^{G,geom}(N,\partial N)}.
  \end{CD}
\end{equation}
By functoriality, $j_*$ is injective (indeed a split injection). Moreover, we
just showed that $\beta_N\alpha_N=\id$. According to the discution above, the
definition of $\beta_N$ for $N$ does not depend on the 
 chosen retraction and we choose $N$ as a retraction of
 itself. In this case, since $\alpha$ is  an isomorphism, the left
 square commutes and therefore $\beta\alpha=\id$ also for $X$, using the 
 equality 
 \begin{equation*}
\begin{split}
  \alpha_Nj_*\beta\alpha[M,E,f] &= \alpha[N, \PD_N^{-1}j_*f_*\PD_M[E], j\circ
  p] \\
  &=j_*p_*\PD_N\PD_N^{-1}j_*f_*\PD_M[E]\\
  &= j_*p_*\PD_M[E] = \alpha_N j_*[M,E,f].
\end{split}
\end{equation*}
Therefore $\beta$ is inverse to $\alpha$ in general, proving our main theorem}.

\begin{appendix}
  
\section{Bott periodicity and Thom isomorphism in equivariant $KK$-theory}
\label{sec:bott-thom}

Bott periodicity and the Thom isomorphism are classical results of
$K$-theory. It is well-known that these isomorphisms can be implemented by
Kasparov multiplication with certain $KK$-equivalences called the \term{Bott
  element} and \term{Thom element}, repectively. Although one finds many
constructions of these elements in the literature, they are often done in a
different context. 
As their relationship is crucial to a proper understanding of vector bundle modification we will sketch the relevant results in this appendix.

%\subsubsection*{Conventions}
Following the usual conventions of analytic $K$-homology
\red{\cite{MR1817560,kasparov75}} and the previous articles \red{\cite{BHS09,BHS}},
$\CliffordC_n = \CliffordC_{0,n}$ is the Clifford algebra of $\complexs^n$ that
is defined so that $e_i e_j + e_j e_i = -2 \delta_{i,j}$ for the standard
basis $(e_i)$. 
We will also need the (isomorphic) Clifford algebra $\CliffordC_{-n} = \CliffordC_{n,0}$ with respect to the negated quadratic form which is commonly used in $KK$-theory \cite{kasparov81}.
The subgroups $\PinC_n$ and $\SpinC_n$ are then defined as usual; {again for each $n\in\integers$, the ones for
$n$ are isomorphic to the ones for $-n$ .}
With these definitions we have $KK_n(A,B) = KK(A, B \hat{\otimes} \CliffordC_n)$ for all $n \in \integers$.

\subsection{Equivariant $\spinC$-structure of the spheres}

A careful analysis of the canonical $\spinC$-structure on $S^n$ is key to the results of this appendix.

Let $\SpinC_{n+1}$ act on the ball $D^{n+1}$ by rotations (i.e. via the canonical homomorphism $\rho\colon  \SpinC_{n+1} \rightarrow SO_{n+1}$). Then the natural $\spinC$-structure of $D^{n+1}$ is also $\SpinC_{n+1}$-equivariant (it is trivial and the group acts by rotation on the base and by left multiplication on the fiber). As the boundary $S^n = \partial D^{n+1}$ is invariant under this action, the equivariant version of the usual boundary construction induces a natural $\SpinC_{n+1}$-equivariant $\spinC$-structure on $S^n$. In the following we will use the ``outer normal vector first'' boundary orientation convention as in \cite[p. 90]{MR1031992} or \cite[3.2]{BHS} but still identify $\SpinC_n \subseteq \SpinC_{n+1}$ by the natural inclusion $\reals^n \subseteq \reals^{n+1}$. Hence the north pole $e_{n+1}$ is stabilized by the rotation action of $\SpinC_n$ and we get the following lemma.

\begin{lemma}
  The $\SpinC_{n+1}$-equivariant principal $\SpinC_n$-bundle of $S^n$ is
  \begin{equation*}
    \SpinC_{n+1} \rightarrow S^n, \quad g \mapsto (-1)^n \rho(g) e_{n+1}
  \end{equation*}
  where the left and right actions are given by multiplication.
\end{lemma}

Let us restrict the left action to $\SpinC_n \subseteq \SpinC_{n+1}$. Then the
hemispheres (which we will denote by $S^n_\pm$) are invariant and a variation
on the argument in \cite[\S 13]{MR0167985} yields the following representation:

\begin{lemma}
  \label{lem:clutching}
  The $\SpinC_n$-equivariant principal $\SpinC_n$-bundle of $S^n$ is $\SpinC_n$-equivariantly isomorphic to the one obtained by glueing the two bundles
  \begin{equation*}
    S^n_\pm \times \PinC_{n,\pm} \rightarrow S^n_\pm
  \end{equation*}
  along the equator via the identification $(x,g) \mapsto (x,(-1)^n xg)$.

{ Here $\PinC_{n,+}=\SpinC_{n}$ and $\PinC_{n,-}$ is the other component of $\PinC_n$.}
\end{lemma}

For every graded $\CliffordC_n$-module $W = W^+ \oplus W^-$ we have a natural isomorphism $\PinC_{n,-} \times_{\SpinC_n} W^+ \cong W^-$.
Hence the even part of the associated spinor bundle on $S^n$ is given by the analogue clutching construction applied to the bundles
\begin{equation*}
  S^n_\pm \times W^\pm \rightarrow S^n_\pm.
\end{equation*}
In particular if $W$ is the standard graded irreducible representation of $\CliffordC_{2k}$ this gives a description of the even part $\Sred{S^{2k}}^+$ of the reduced spinor bundle of an even-dimensional sphere $S^{2k}$. We will later be interested in its dual or, equivalently, its conjugate. It is given by the same clutching construction applied to the conjugate $\CliffordC_{2k}$-module $\overline{W}$, and from the action of the complex volume element we see that it is isomorphic to $W$ precisely if $k$ is even.

\bigskip

Let us now turn to computing equivariant indices for \emph{even-dimensional} spheres. It is clear that the index of its equivariant $\CliffordC_{2k}$-linear Dirac operator
\begin{equation*}
    \collapse_* [S^{2k}]
  = \collapse_* \partial [D^{2k+1}]
  = \partial \collapse_* [D^{2k+1}]
\end{equation*}
vanishes as we have factored over $R_{2k+1}(\SpinC_{2k+1})=0$ (using
naturality of the boundary map). {Recall that, for a compact Lie group $G$, $R(G)$ is
the complex representation ring, which is canonically isomorphic to
$K^0_G(*)$, and $R_n(G):=K^n_G(*)$.} This argument in fact only depends on the
Dirac bundle over the sphere being induced by the boundary construction from a
Dirac bundle over the ball. We conclude that the index still vanishes if we
consider instead the \emph{reduced} spinor bundle $\Sred{S^{2k}}$ twisted with
the pullback $E$ of a representation in $R(\SpinC_{2k+1})$ (which of course
extends over the ball). 

\bigskip

On the other hand, recall that for every closed even-dimensional $\spinC$-manifold $M$, Clifford multiplication induces isomorphisms of Dirac bundles $\Cliff^\complexs(M) \cong \End(\Sred{M}) \cong \Sred{M} \hat{\otimes} \Sreddual{M}$. If we identify $\Cliff^\complexs(M)$ with the complexified exterior bundle then an associated Dirac operator is given by the de Rham operator (cf. \cite[11.1.3]{MR1817560}). There is a canonical involution on $\Cliff^\complexs(M)$ induced by \emph{right} Clifford multiplication with $i^k E_1 \cdots E_{2k}$ where $(E_i)$ is any oriented local orthonormal frame; let us designate its positive eigenbundle by $\Cliff^\complexs_{\frac{1}{2}}(M)$. It is invariant under the de Rham operator, and the above maps restrict to an isomorphism
\begin{align*}
  \Cliff^\complexs_{\frac{1}{2}}(M) \cong \Sred{M} \otimes \Sreddualeven{M}.
\end{align*}
In particular, the above construction applies to the even-dimensional sphere
$M = S^{2k}$ and works equivariantly if we equip the exterior bundle with the
action induced by $\rho$. Since the de Rham operator is rotation-invariant we
can still use it as our Dirac operator. Its kernel consists precisely of the
harmonic forms, hence in view of the cohomology of $S^{2k}$ it is spanned by a
$0$-form and a $2k$-form (which are rotation-invariant and interchanged by the
involution). It follows that after restricting to the positive eigenbundle the
kernel is just the one-dimensional trivially-graded trivial representation. In
other words, 
\begin{equation*}
  \collapse_* [\Sred{S^{2k}} \otimes \Sreddualeven{S^{2k}}] = 1.
\end{equation*}
Expressing twisted indices as Kasparov products (cf. \cite[24.5.3]{MR1656031}) we have
\begin{equation}
  \label{eqn:indices}
  ([\Sreddualeven{S^{2k}}] - [E]) \otimes_{\red{C(S^{2k})}} [\Sred{S^{2k}}] = 1 \in R(\SpinC_{2k+1})
\end{equation}
for every pullback $E$ of a representation in $R(\SpinC_{2k+1})$.

\subsection{Topological Bott periodicity}
We will now construct equivariant Bott elements $b_{2k} \in K^0_G(\reals^{2k})$ where $G$ is a compact group acting \term{spinorly} on $\reals^{2k}$ (i.e. the action factors over a continuous homomorphism $G \rightarrow \SpinC_{2k}$). Let us identify $\reals^{2k}$ $G$-$\spinC$-structure-preservingly with an open subset of its one-point compactification $S^{2k}$ via stereographic projection from the south pole.
If we now use the split short exact sequence
\begin{equation*}
  \xymatrix
  {
    0 \ar[r] & K^0_G(\reals^{2k}) \ar[r]^{\incl_*} & K^0_G(S^{2k}) \ar[r] & \ar[l] K^0_G(*) = R(G) \ar[r] & 0
  }
\end{equation*}
to pull the south pole fiber of $F_0 := \Sreddualeven{S^{2k}}$ back to a bundle $F_0^\infty$ over the entire sphere then by exactness there is a unique preimage of $[F_0] - [F_0^\infty]$ which we will call the \term{Bott element} $b_{2k} \in K^0_G(\reals^{2k})$. Using equation $(\ref{eqn:indices})$ we get
\begin{equation*}
    b_{2k} \otimes_{C_0(\reals^{2k})} [\Sred{\reals^{2k}}]
  = b_{2k} \otimes_{C_0(\reals^{2k})} \incl^* [\Sred{S^{2k}}]
  = \incl_* b_{2k} \otimes_{C(S^{2k})} [\Sred{S^{2k}}]
  = 1.
\end{equation*}
The following version of Atiyah's rotation trick \cite{MR0228000} now allows us to establish that $b_{2k}$ is in fact a $KK$-equivalence:
\begin{lemma}
  Let $b \in K^0_G(\reals^n)$ and $D \in K_0^G(\reals^n)$ satisfy $b \otimes_{C_0(\reals^{n})} D = 1$.
  Then they are already $KK$-equivalences inverse to each other.
\end{lemma}
\begin{proof}
\green{Let, in the following,
  $\tensor$ denote the external Kasparov product, and $\tensor_A$
  the composition Kasparov product.} \green{Recall first that for
  $G$-$C^*$-algebras $A,A',B,B'$ and $z\in KK^G(A,A')$, $z'\in KK(B,B')$ we have the
  following commutativity {of the exterior Kasparov product:
  \begin{equation}\label{eq:tensoridentity}
%  z\tensor z'=   
 \tau_{B}(z)\tensor_{A'\tensor B} \tau_{A'}( z') =
  \tau_{A}(z')\tensor_{A\tensor B'} \tau_{B'}(z) \in KK^G(A\tensor B,A'\tensor B').
  \end{equation}
(where for   $G$-$C^*$-algebras $A,B$ and $D$, $\tau_D\colon KK^G(A,B)\to KK^G(A\tensor
D,B\tensor D)$ is external tensor product with $\kk{1_D}$).}}

  As the rotation $(x,y) \mapsto (y,-x)$ is $G$-equivariantly homotopic to the
  identity we get (the third identity in)
  \begin{equation*}
    \begin{split}
      D \otimes_{\complexs} b & =\green{\left(
        D\tensor \kk{1_\complexs}\right) \tensor_\complexs
      \left(\kk{1_\complexs} \tensor b \right)} \\
     &\green{\stackrel{~\eqref{eq:tensoridentity}}{=}}\green{
  \left(\kk{1_{C_0(\reals^n)}} \otimes b\right) \otimes_{C_0(\reals^{2n})} \left(D
    \otimes \kk{1_{C_0(\reals^{n})}}\right)}\\
    & = \green{ \left(b \otimes \Theta\right) \otimes_{C_0(\reals^{2n})}
      \left(D \otimes \kk{1_{C_0(\reals^{n})}}\right)}\\
     & =\green{\left(b \otimes_{C_0(\reals^{n})} D\right) \otimes \left(\Theta \otimes_{C_0(\reals^{n})}
 \kk{1_{C_0(\reals^{n})}}\right)}\\
 & =  \green{\kk{1_\complexs} \otimes \Theta=\Theta}
%
%
%  & \left(D\tensor_\complexs \kk{1_\complexs}\right) \tensor
%      \left(\kk{1_\complexs} \tensor_\complexs b\right) \\
%      &=\green{\left( \kk{1_{C_0(\reals^{2n})}} \tensor_{C_0(\reals^{2n})} D\right)
%      \tensor \left(b\tensor_{C_0(\reals^{2n})} \kk{1_{C_0(\reals^{2n})}}\right)}\\
%      &=\green{ (b \otimes \kk{1_{C_0(\reals^n)}}) \otimes_{C_0(\reals^{2n})} (\kk{1_{C_0(\reals^n)}} \otimes D)}\\
%      &= (\Theta \otimes b) \otimes_{C_0(\reals^{2n})} (\kk{1_{C_0(\reals^n)}} \otimes D)
%       = \Theta
    \end{split}
  \end{equation*}
  where $\Theta$ is the $KK$-involution corresponding to $x \mapsto -x$. 

  It follows that $b$ also has a left $KK$-inverse (and, in fact, $\Theta = 1$).
\end{proof}

\begin{corollary}[Topological Bott periodicity]
  Let $G$ be a compact group acting spinorly on $\reals^{2k}$.
  Then the associated Bott element $b_{2k} \in K^0_G(\reals^{2k})$ and
  the reduced fundamental class $[\Sred{\reals^{2k}}] \in K_0^G(\reals^{2k})$
  are $KK$-equivalences inverse to each other.
\end{corollary}

\subsection{Analytical Bott periodicity}
We will now derive an analytic cycle for the Bott element. Clearly, the
difference bundle $[F_0] - [F_0^\infty]$ is represented by the 
\green{$KK_G(\complexs,C(S^{2k}))$-cycle}
\begin{equation}
  \label{eqn:topo-bott}
  (\Gamma(F_0) \oplus \Gamma(F_0^\infty)^{\op}, \mul_\complexs, 0).
\end{equation}
Consider the description of $F_0$ from the discussion after Lemma
\ref{lem:clutching}. Evidently, $F_0^\infty$ can be described by a similar
clutching construction given by using the southern hemisphere representation
over both hemispheres and glueing using the identity. We can thus define an
operator $T'$ acting on even (odd) sections by pointwise Clifford
multiplication with plus (minus) the first $n$ coordinates of the respective
base point on the upper hemisphere and by the identity on the lower
hemisphere, and the linear path gives a homotopy to the cycle
\begin{align*}
  (\Gamma(F_0) \oplus \Gamma(F_0^\infty)^{\op}, \mul_\complexs, T').
\end{align*}
Now we can restrict to the open upper hemisphere via the homotopy given by the Hilbert $C([0,1], C(S^{2k}))$-module
\begin{equation*}
  \{
    f \in C([0,1], \Gamma(F_0) \oplus \Gamma(F_0^\infty)^{\op}) :
    f(0) \in \Gamma_0(F_0|_{\mathring{S}^{2k}_+}) \oplus \Gamma_0(F_0^\infty|_{\mathring{S}^{2k}_+})^{\op}
  \}.
\end{equation*}
\green{Note that all conditions on a Kasparov triple are satisfied because $T'$
  is an isomorphism on $\mathring{S}^{2k}_-$.}

Identifying $\reals^{2k}$ with the open upper hemisphere via $x \mapsto (x,1)/\sqrt{1 + ||x||^2}$ (which is equivariantly homotopic to our previous identification) we conclude that the Bott element is given by the cycle
\begin{equation*}
  b_{2k} = [C_0(\reals^{2k}, \overline{W}), \mul_\complexs, T] \in K^0_G(\reals^{2k})
\end{equation*}
with the obvious Hilbert $C_0(\reals^{2k})$-module structure and where
$T$ acts by Clifford multiplication with $\pm x/\sqrt{1 + ||x||^2}$
\red{on  the conjugate $\overline{W}$ of the standard graded
  irreducible representation of $\CliffordC_{2k}$} .

Chasing the relevant definitions in \cite[Sections 2 and 5]{kasparov81} one finds that the formal periodicity isomorphism $KK^G_0 \cong KK^G_{-2k}$ amounts to tensoring with the standard graded irreducible $\CliffordC_{-2k}$-module. Thus the image of $b_{2k}$, after applying a unitary equivalence, is
\begin{align*}
  \beta_{2k} := [C_0(\reals^{2k}, \CliffordC_{-2k}), \mul_\complexs, T_{2k}] \in K^{2k}_G(\reals^{2k})
\end{align*}
where $T_{2k}$ acts by Clifford multiplication with $x/\sqrt{1 + ||x||^2}$. This is the classical cycle of the Bott element due to Kasparov
\cite[Paragraph 5]{kasparov81}.
As the image of $[\Sred{\reals^{2k}}]$ under formal periodicity of $K$-homology is of course the fundamental class of $\reals^{2k}$ we have proved the following result for $n = 2k$.

\begin{proposition}[Analytical Bott periodicity]
  Let $G$ be a compact group acting spinorly on $\reals^n$.  
  Then the Bott element
  \begin{equation*}
    \beta_n := [C_0(\reals^n, \CliffordC_{-n}), \mul_\complexs, T_n] \in K^n_G(\reals^n)
  \end{equation*}
  (where $T_n$ is the Clifford multiplication operator defined as above)
  and the fundamental class $[\reals^n] \in K_n^G(\reals^n)$ are $KK$-equivalences inverse to each other.
\end{proposition}

A purely analytic argument can be used to show that it holds in arbitrary dimensions (see e.g. \cite[5.7]{kasparov81}).
Note that we have $\beta_n = (\beta_1)^n$ (for appropriate group actions); this follows readily from the product formula for fundamental classes.

\subsection{Thom isomorphism}

Let $G$ be a compact topological group and $\pi_W\colon W \rightarrow X$ be a $G$-$\spinC$-vector bundle of dimension $n$ over a compact $G$-space $X$ with principal $\SpinC_n$-bundle $P$.
Then $P/{\SpinC_n} \cong X$ and $P \times_{\SpinC_n} \reals^n \cong W$ and Kasparov's induction machinery from \cite[3.4]{kasparov88} is applicable.
Let us define the \term{Thom element} $\beta_W$ as the image of $\beta_n$ under the composition
\begin{equation*}
  \begin{split}
   &K_{G \times \SpinC_n}^n(\reals^n) =
    RKK^{G \times \SpinC_n}_{-n}(*; \complexs, C_0(\reals^n))\\
  \stackrel{\collapse^*}{\longrightarrow}
    &RKK^{G \times \SpinC_n}_{-n}(P; \complexs, C_0(\reals^n))
  = \mathcal{R}KK^{G \times \SpinC_n}_{-n}(P; C_0(P), C_0(P \times \reals^n))\\
  \stackrel{\lambda^{\SpinC_n}}{\longrightarrow}
    &\mathcal{R}KK^G_{-n}(X; C(X), C_0(W))
  \longrightarrow
    KK^G_{-n}(C(X), C_0(W))
  \end{split}
\end{equation*}
Naturality of these operations shows that $\beta_W$ is a $KK$-equivalence; its inverse is given by the image of the fundamental class of $\reals^n$ under the analogous composition.
Chasing definitions, we find that
\begin{equation*}
  \beta_W = [C_0(P \times \reals^n, \CliffordC_{-n})^{\SpinC_n}, \mul_{C(P)^{\SpinC_n}}, T'_n]
\end{equation*}
where $T'_n$ is the equivariant operator acting by Clifford multiplication with the second coordinate. Denoting the Connes-Skandalis spinor bundle $P \times_{\SpinC_n} \CliffordC_{-n}$ of $W$ by $\Sfull{W}^\text{CS}$ (cf. \cite{MR775126}) we get the following result:

\begin{proposition}[Analytical Thom isomorphism]
  Let $G$ be a compact group and $\pi_W\colon W \rightarrow X$ a $G$-$\spinC$-vector bundle over a compact space $X$. Then the Thom element
  \begin{equation*}
    \beta_W = [\Gamma(\pi_W^*(\Sfull{W}^\text{CS})), \mul_{C(X)}, T_W] \in KK_{-n}(C(X), C_0(W))
  \end{equation*}
  (where $T_W$ is the operator of pointwise Clifford multiplication with
  \green{the base point in $W$}) is a $KK$-equivalence.
\end{proposition}

Let us now consider the even case $n = 2k$. Then we can perform the same construction with the reduced Bott element $b_{2k}$ and the resulting Thom element $b_W$ is just the image of $\beta_W$ under Clifford periodicity. Clearly, $b_W$ is given by the obvious cycle using the reduced spinor bundle $\Sred{W}^\text{CS}$. If on the other hand we start with the topological description (\ref{eqn:topo-bott}) of the Bott element, we find that
\begin{equation*}
  b_W = [\Gamma(P \times_{\SpinC_{2k}} F_0) \oplus \Gamma(P \times_{\SpinC_{2k}} F_0^\infty)^{\op}, \mul_{C(X)}, 0]
\end{equation*}
where $F := P \times_{\SpinC_{2k}} F_0$ is interpreted as a vector bundle over the sphere bundle $Z \cong P \times_{\SpinC_{2k}} S^{2k}$.
We have seen before that $F_0$ can be described using an \emph{equivariant} clutching construction. Consequently, the associated bundle $F$ also arises from a clutching construction and it is easy to see that it is precisely the one used for vector bundle modification:

\begin{proposition}[Topological Thom isomorphism]
  \label{prop:thom-topo}
  The reduced Thom element has the representation
  \begin{equation*}
    b_W = [\Gamma(F) \oplus \Gamma(F^\infty)^{\op}, \mul_{C(X)}, 0] \in KK_0(C(X), C_0(W))
  \end{equation*}
  where $F$ is the bundle over the sphere bundle $Z$ from Definition \ref{def:modification} and where $F^\infty$ is the pullback of the north pole restriction of $F$ back to $Z$.
\end{proposition}

\section{Analytic Poincar\'e duality and Gysin maps}

\subsection{Construction of the Gysin element for closed manifolds}
\label{sec:Gysin_for_closed}

Let $f\colon M \rightarrow N$ be a smooth $G$-map between two compact $G$-$\spinC$-manifolds without boundary. We describe the construction of the (functorial) Gysin element $f! \in KK^G_{\dim M - \dim N}(C(M), C(N))$ which implement the Gysin maps $f!\colon K^G_*(M) \rightarrow K^G_{*+\dim N-\dim M}(N)$.

By functoriality and since every smooth $G$-map $f\colon  M \rightarrow N$ between compact $G$-$\spinC$-manifolds can be written as the composition of the embedding $M \rightarrow M \times N, m \mapsto (x,f(x))$ with the canonical projection $\pi_2\colon  M \times N \rightarrow N$ it suffices to describe the Gysin element associated to an equivariant embedding and to $\pi_2$.

The Gysin element of the projection $\pi_2$ is just the element $(\pi_2)! = \tau_{C(N)} [M]$ obtained by tensoring the fundamental class of $M$ with $C(N)$.
If $f\colon  M \rightarrow N$ is an equivariant embedding of compact $G$-$\spinC$-manifolds then its normal bundle $\nu_M$ is canonically $G$-$\spinC$ and, after fixing a $G$-invariant metric on $N$, can be considered as a $G$-invariant open tubular neighborhood of $N$.
The Gysin element of $f$ is then $f! = \beta_{\nu_M} \otimes_{C_0(\nu_M)} \kk{\theta_M}$ where $\theta_M$ is the equivariant inclusion $C_0(\nu_M) \subseteq C(N)$.

\subsection{Gysin and Poincar\'e duality}
\label{sec:Gysin_and_PD}

\begin{proof}[Proof of Lemma \ref{lem-gysin}, case $\boundary
  M=\emptyset=\boundary N$.]
Let us denote by $[f]$ the element of $KK^G_*(C(N),C(M))$
corresponding  
to the  morphism $C(N)\to C(M); h \mapsto h\circ f$. Then the  commutativity
of the diagram ammouts to prove that 
\begin{equation}\label{equ-gysin} \iota_N(x\otimes
  f!)=[f]\otimes\iota_M(x)\otimes f! 
\end{equation}
for all $x$ in $K_G^*(M)$. Namely, using this equality,
we have 
\begin{equation*}
\PD_N(x\otimes f!)=\iota_N(x\otimes f!)\otimes [N]
=[f]\otimes\iota_M(x)\otimes f!\otimes[N].
\end{equation*}
Since  $[N]$ is the Gysin element corresponding to \green{the map} $N\to\{*\}$,
we \green{get that} $f!\otimes[N]=[M]$ and hence that
\begin{equation*}
\PD_N(x\otimes f!)
=[f]\otimes\iota_M(x)\otimes [M]
= f_*(\PD_M(x)).
\end{equation*}
Let us now prove Equation \ref{equ-gysin}. Since $f$ can be written as the composition of
 an embedding and of the projection $\pi_2\colon M\times N\to N$, it
 is
 enough  by using the
functoriality in K-homology and the composition rule for Gysin elements
to check this for an embedding and for $\pi_2$.

\green{We start with $\pi_2$. Fix $x\in K_G^*(M\times N)$.
% Recall now first that for
%   $G$-$C^*$-algebras $A,A',B,B'$ and $z\in KK^G(A,A')$, $z'\in KK(B,B')$ we have the
%   following commutativity {of the exterior Kasparov product:
%   \begin{equation*}
%     \tau_{B}(z)\tensor \tau_{A'}( z') = \tau_{A}(z')\tensor \tau_{B'}(z) \in KK^G(A\tensor B,A'\tensor B').
%   \end{equation*}
% (where for   $G$-$C^*$-algebras $A,B$ and $D$, $\tau_D\colon KK^G(A,B)\to KK^G(A\tensor
% D,B\tensor D)$ is tensorising by $D$).}
  Recall that {$\pi_2!=\tau_{C(N)}([M])$ and
    $[\pi_2]=\tau_{C(N)}([p])$, for $p\colon M\to\{*\}$ }, and that we can write}
  {\begin{equation*}
\iota_{M\times N}(x)=\tau_{C(M\times N)}(x) \tensor \mu_{M\times N},
\end{equation*}}
  where  $\mu_{M\times
  N}\colon C(M\times N)\tensor C( M\times N)\to C(M\times N)$ is the
  multiplication. Then, \green{using also \eqref{eq:tensoridentity}}, {
  \begin{eqnarray*}
      [\pi_2]\tensor\iota_{M\times N}(x)\tensor \pi_2! & =& \tau_{C(N)}[p]\tensor\tau_{C(M\times N)}(x)
      \tensor [\mu_{M\times N}]\tensor \pi_2! \\
 & = &\tau_{C(N)}\left([p]\tensor\tau_{C(M)}(x)\right)
      \tensor [\mu_{M\times N}]\tensor \pi_2! \\
 & =& \tau_{C(N)}\left(x\tensor \tau_{C(M\times N)}[p]\right)
\tensor [\mu_{M\times N}]\tensor \pi_2! \\
& =& \tau_{C(N)}(x)\tensor \tau_{C(N\times M\times N)}[p]
\tensor [\mu_{M\times N}]\tensor\pi_2! \\
& =& \tau_{C(N)}(x)\tensor \tau_{C(M)}([\mu_N]) \tensor\tau_{C(N)}([M]) \\
& =& \tau_{C(N)}(x)\tensor \tau_{C(N\times N)}([M])\tensor [\mu_N] \\
& =& \tau_{C(N)}(x\tensor \tau_{C(N)}([M]))\tensor [\mu_N] \\
& =& \tau_{C(N)}(x\tensor  \pi_2!)\tensor [\mu_N] \\  
     & =& \iota_N(x\tensor \pi_2!).
\end{eqnarray*}}

Recall that, if
$x$ is an element in $K_G^*(M)$, then $\iota_M(x)$ is the element of
$KK^G_*(C(M),C(M))$ obtained from  any K-cycle representing $x$ by
noticing that $C(M)$ being commutative, the right action is also a
left action.
Since   $x\otimes f!=[p]\otimes\iota_M(x)\otimes f!$, where $[p]$ is the
element of $K^*_G(M)$ corresponding to the inclusion
$\C\hookrightarrow C(M)$, we can see that if $(\phi,T,\xi)$ is a
K-cycle representing $\iota_M(x)\otimes f!$, then $\iota_N(x\otimes
  f!)$ can be represented by the K-cycle $(\phi',T,\xi)$ where $\phi'$
  is equal to the (right) action of $C(N)$  on $\xi$. Thus we only have to
  check that $(\phi',T,\xi)$ and $(\phi\circ f,T,\xi)$ represent
  the same class in $KK^G_*(C(N),C(N))$.
%  Noticing first that
%   $\iota_M$ being  the
%   composition $K_G^*(M)\stackrel{\cong}{\to}K_G^*(\C,C(M))\to
%   KK^G_*(C(M),C(M))$, where the second map is just forgetting the
%   $C(M)$-structure, the $K$-cycle $(\phi,T,\xi)$ can be contructed
%   out a K-cycle representing $x$  in
%   such a way that the representation $\phi$ is in fact inherited from the
% representation $\psi$ of $C(M)$ on the K-cycle defining $f!$.

% When $f$ is a submersion, let $(\phi,T,\xi)$ be the K-cycle defining
% $f!$. Then $\xi=(\xi_z)_{z\in N}$ is 
% the family of $L^2$-spinor on $M_x=f^{-1}(\{x\})$,  the $C(N)$-action being given by pointwise
% multiplication.
% The left action of $C(M)$ being fiberwise induced at $z\in N$ by the
% restriction to $C(M)\to C(M_z);\, h\mapsto h_{/M_z}$ since $g\circ
% f_{M_z}=g(z)$, the left action of $C(N)$ on the continuous family
% of $L^2$-spinor on $(M_z=f^{-1}(\{z\})_{z\in N})$ induced by thewill
% morphism $C(N)\to C(M);\, h\mapsto h\circ f$ is the fiberwise action
% of $C(N)$.

Since $f$ is an embbeding, the KK-element $f!$ can be represented by the KK-cycle
$(\phi_{\nu_M},q_{\nu_M}^*\xi_{\nu_M},T_{\nu_M})$ where $q_{\nu_M}^*\xi_{\nu_M}$
is viewed as a $C(N)$-Hilbert module via the inclusion
$C_0(\nu_M)\hookrightarrow C(N)$ and where $\phi_{\nu_M}$ is the
representation induced by $\phi_0\colon C(M)\to C_b(\nu_M);h\mapsto h\circ
q_{\nu_M}$. Thus we can choose the $K$-cycle $(\phi,T,\xi)$
representing $\iota_M(x)\otimes f!$ in such a
way that 
\begin{itemize}
\item $\xi$ is in fact a $C_0(\nu_M)$-Hilbert module (by associativity of the
  Kasparov product);
\item  $T$ commutes with the action of $C_b(\nu_M)$ viewed as the
  multiplier algebra of $C_0(\nu_M)$ (use an approximate unit and the
  continuity of $T$, observe that $The=heT=hTe$ for all $h\in C_b(\nu_M)$,
  $e\in C_0(\nu_M)$);
\item the  $C(M)$-structure is induced by $\phi_0$.
\end{itemize}
The maps $\nu_M\to\nu_M;v\mapsto tv$ for $t$ in $[0,1]$ provide a
homotopy between  $\phi_0\circ f$ and  the restricton map $C(N)\to
C_b(\nu_M)$ and this homotopy commutes  with $T_{\nu_M}$. But the
restriction map corresponds precisely to the $C(N)$-Hilbert
module structure on  $q_{\nu_M}^*\xi_{\nu_M}$, and hence we get the
result.
\end{proof}

\subsection{Gysin and Poincar\'e duality if $\boundary M\ne \emptyset$}

Our key tool to study $G$-manifolds with boundary is the double.

For a manifold  $X$  with boundary $\partial X$, let us define the double
$\db X$  of $X $ to be  the manifold obtained by
identifying the two copies of the boundaries $\partial X$ in $X\amalg X$. To
distinguish the two copies, we write $\db X=X\cup X^-$. Let
$p_X\colon \db X\to X$ be the map obtained by identifying the two 
copies of $X$. 
 Let $\j_X\colon X\to\db X$ be the map induced by the
inclusion of the first factor of $X\amalg X$, and let us set
$g_X=\j_X\circ p_X$.
It is
straightforward to check that if $X$ is a \spg compact manifold with
boundary $\partial X$,
then $\db X$ is a \spg  compact manifold without boundary. The given
orientation or $G$-$\spin^c$ structure on the first copy of $X$ and the
negative structure on the second copy $X^-$ together define canonically a
$G$-$\spin^c$ structure on $\db X$. Note that $p_X\circ j_X=\id_X$. Therefore,
the exact sequence 
\begin{equation*}
0\to C_0(X^-\setminus\boundary X) \xrightarrow{\iota_X}
C(\db X)\xrightarrow{j_X^*} C(X)\to 0
\end{equation*}
has a split, and we get induced split exact
sequences in K-theory and K-homology. Note that in general there is no
split
of $\iota_X$ by algebra homomorphisms, but the corresponding split in K-theory
and K-homology of course exists nonetheless. We use the corresponding sequence
and split with the roles of $X$ and $X^-$ interchanged.
will
We now state the workhorse lemma for the extension of the treatment of Gysin
homomorphism from closed manifolds to manifolds with boundary.

\begin{lemma}\label{lem:workhorse}
  Poincar\'e duality for $M$ is a direct summand of Poincar\'e duality for
  $\db M$, i.e.~the following diagram commutes, if $M$ is a compact
  $G$-$\spin^c$ manifold with boundary.
  \begin{equation}\label{eq:workhorsesquare}
    \begin{CD}
      K^*_G(M) @>{p_M^*}>> K^*(\db M) @>{j_M^*}>> K^*(M)\\
      @VV{\PD_M}V @VV{\PD_{\db M}}V @VV{\PD_M}V\\
        {K_{n-*}^G(M,\partial M)} @>{s}>> K_{n-*}^G(\db M) @>{(\iota_M)^*}>>
      {K_{n-*}^G(M,\partial M)}. 
    \end{CD}
  \end{equation}
  Here,  $s$ is the K-homology split
  mentioned above, and $n=\dim(M)=\dim(\db M)$.
\end{lemma}
\begin{proof}
  This result is certainly well known. For the convenience of the reader, we
  give a proof of it here.

  We use the following alternative description of the Poincar\'e duality
  homomorphism. For a compact manifold $M$ (possibly with boundary) it is the
  composition 
  \begin{equation*}
    KK(\{*\},M)\xrightarrow{\tau_{C_0(M^\circ)}} KK(M^\circ,M^\circ\times M) \xrightarrow{\mu}
    KK(M^\circ,M^\circ)\xrightarrow{\tensor[M^\circ]} KK(M^\circ,\{*\}).
  \end{equation*}
  Here we abbreviate $KK$ for $KK^G_*$ (and ask the reader to add the correct
  grading), and write {$KK(X,Y)=KK(C_0(X),C_0(Y))$ for two spaces $X,Y$}, $\mu$
  is the map induced by the multiplication $C_0(M\times M^\circ)=C(M)\tensor
  C_0(M^\circ)\to C_0(M^\circ)$.
will
  Naturality of KK-theory and of the fundamental class (under inclusion of
  open submanifolds) now
  gives the following commutative diagram, writing $N=\db M$
{\small  \begin{equation*} 
    \begin{CD}
      KK(\{*\},M) @>{\tau_{C_0(M^\circ)}}>> KK(M^\circ,M^\circ\times M) @>>{\mu}> KK(M^\circ,M^\circ)\\
     @AA{j^*}A  @AA{j^*}A @|\\
      KK(\{*\},N)@>{\tau_{C_0(M^\circ)}}>> KK(M^\circ,M^\circ \times N) @>\mu>>
   KK(M^\circ,M^\circ) @>>{\tensor [M^\circ]}> KK(M^\circ,\{*\})\\
      @VVV  @VV{\iota_*}V @VV{\iota_*}V @|\\
      KK(N,N\times N) @>{\iota_M^*}>> KK(M^\circ,N\times N) @>\mu>> 
      KK(M^\circ, N) @>{\tensor [N]}>> KK(M^\circ,\{*\})\\
      @| @AA{\iota^*}A @AA{\iota^*}A @AA{\iota^*}A\\
      KK(N,N\times N) @= KK(N,N\times N) @>\mu>> 
      KK(N, N) @>{\tensor [N]}>> KK(N,\{*\})\\
    \end{CD}
  \end{equation*}
}
  Walking around the boundary of this diagram shows that the right square of
  \eqref{eq:workhorsesquare} is commutative.

  The commutativity of the left square of \eqref{eq:workhorsesquare} is more
  difficult to show, in 
  particular since $s$ is not induced from an algebra homomorphism. However,
  from what we have just seen we can conclude that
  \begin{equation}  \label{eq:withiota}
    \iota^*\PD_{\db M}p^* = \PD_M{j^*p^*}\stackrel{{p\circ j=\id_M}}{=}\PD_M = \iota^*s\PD_M.
  \end{equation}

%   we have the following commutative diagram, writing $N:=\db M$will
%   \begin{equation*}\small
%     \begin{CD}
%       KK(\{*\},M)@>>> KK(M^\circ,M^\circ\times M) @>{\mu}>> KK(M^\circ,M^\circ)
%       @>{\tensor [M]}>> KK(M^\circ,\{*\})\\
%       @VV{p^*}V @VV{p^*}V @VV{\iota_*}V  @|\\
%       KK(\{*\},N) @>>> KK(M^\circ,M^\circ\times N) @>{\mu=\iota\circ\mu}>>
%       KK(M^\circ,N) @>{\tensor [N]}>> KK(M^\circ, \{*\})\\
%       @VVV   @VV{\iota_*}V   @| @|       \\
%       KK(N,N\times N) @>{\iota^*}>> KK(M^\circ,N\times N) @>\mu>>
%       KK(M^\circ,N) @>{\tensor[N]}>> KK(M^\circ, \{*\})\\
%       @|      @AA{\iota^*}A  @AA{\iota^*}A  @AA{\iota^*}A\\
%       KK(N,N\times N) @= KK(N,N\times N) @>{\mu}>>
%       KK(N,N) @>{\tensor[N]}>> KK(N,\{*\})
%     \end{CD}
%   \end{equation*}
%   We conclude that 
%   \begin{equation} 
% \iota^*s\PD_M=\PD_M=\iota^*\PD_{N}p^*.  
% \end{equation}
The section $s$ is  
  characterized by the properties $\iota^*s=\id$ and $sp_*=0$. Therefore, to
  be allowed to ``cancel'' $\iota^*$ in Equation \eqref{eq:withiota} we have
  to show that $\PD_{\db M}p^*$ maps to the image of $s$, i.e.~to the kernel
  of $p_*$. We must show that
  \begin{equation}\label{eq:null}
 0=p_*\PD_N p^*\colon KK(\{*\},M)\to KK({N},\{*\}).
 \end{equation} 
  The relevant groups, namely $K^*(M)$, $K^*(\db M)$, $K_*(\db M)$,
  $K_*(M^\circ)$ all are $K^*(M)$-modules, and all homomorphisms are
  $K^*(M)$-module homomorphisms. The module structure on $K^*(M)$ is induced
  via the ring structure of $K^*(\db M)$ and the map $p^*$. $K_*(\db M)$ is a
  $K^*(\db M)$-module via the cap product, and via $p^*$ it therefore also
  becomes a $K^*(M)$-module; the cap product also gives the $K^*(M)$-module
  structure on $K_*(M^\circ)$. 

As $K^*(M)$ is generated by $1$ as a
  $K^*(M)$-module, Equation \eqref{eq:null} follows ifwill 
  \begin{equation*}
0=p_*\PD_{\db M}p^*1=
  p_*[\db M]. 
\end{equation*}

To see this, remember that every double of a manifold with boundary is
canonically a boundary, namely $\db
M= \boundary(Y:=(M\times [-1,1]/\sim))$, where the equivalence relation is
generated by $(x,t)\sim (x,s)$ is $x\in\boundary M$ and $s,t\in
[-1,1]$. Observe that this construction is valid in the world of $G$-$\spin^c$
manifolds. Note that $p_M\colon \db M\to M$ extends to $$P\colon Y=(M\times
[-1,1]/\sim) \to (M\times [0,1]/\sim); (x,t)\mapsto (x,\abs{t}).$$ From the long
exact sequences of the pairs $(Y,\db M)$ {and $(M\times [0,1]/\sim,M{\times\{1\}})$}, we have the following commutative
diagram 
\begin{equation*}
  \begin{CD}
    K_{\dim M+1}(Y^\circ) @>{P_*}>> K_{\dim M+1}(C_0((M\times
    [0,1]/\sim)\setminus M\times \{1\})) =\{0\}\\
    @VV{\boundary}V @VV{\boundary}V\\
    K_{\dim M}(\db M) @>{p_*}>> K_{\dim M}(M).
  \end{CD}
\end{equation*}
In this diagram, $[Y^\circ]$ {is, according to Lemma 
\ref{lem-boundary}, mapped  under the boundary} to $[\db
M]$. Therefore, by naturality, $p_*[\db M]=\boundary P_*([Y^\circ])$. However,
\begin{equation*}
(M\times [0,1]/\sim) \setminus M\times \{1\}= M^\circ\times [0,1),
\end{equation*}
and
$C_0(M^\circ\times [0,1))$ is $G$-equivariantly contractible, hence its
equivariant K-homology vanishes. The assertion follows.
\end{proof}

\begin{definition}\label{def:Gysin_boundary}
  Let now $f\colon M\to N$ be a $G$-equivariant continuous map between
  $G$-$\spin^c$ manifolds with boundary such that $f(\boundary M)\subset
  \boundary N$. Then we define $f!\colon K^*_G(M)\to K_G^{*+n-m}(N)$ as the
  composition
  \begin{equation*}
    f!=\PD_N^{-1} f_* \PD_M.
  \end{equation*}
\end{definition}

\begin{remark}\label{rem:other_defs}
  Note that this is consistent with the definition for closed manifolds and
  smooth maps by the
  considerations of Section \ref{sec:Gysin_and_PD}. Lemma \ref{lem-gysin}
  holds in the general case by definition.

However, at least in special situations, we can also define the Gysin map
geometrically. Let, for example, $M$ be a  \spg compact manifold with
boundary, let $W$ be a \spg vector bundle 
over $M$, let $Z$ be the manifold obtained from vector bundle
modification with respect to $W$ and, as above,  let $\pi\colon Z\to M$ and
$s\colon M\to Z$ will
be the canonical projection or the ``north pole'' section of
$\pi$, respectively. The vector bundle $W$ is the normal vector bundle  of $M$ in $Z$ (with respect to the
embedding 
$s$) and is therefore   a $G$-invariant tubular open neighbourhood of
$M$. 

  We can then define the Gysin element $s!\in KK^G(C(M),C(Z))$ associated to
  $s$ as we did for manifold without boundary by $s!=\beta_{W}\otimes
  [\theta_M]$, where $\theta_M\colon C_0(W)\to C(Z)$ is the morphism induced
  by the inclusion of ${W}$ into $Z$. The Gysin homomorphism can then be
  defined correspondingly.

  With arguments similar to those of Section \ref{sec:Gysin_and_PD} we can
  show that with this definition Lemma \ref{lem-gysin} holds, so that our
  Definition \ref{def:Gysin_boundary} is consistent with the geometric
  one. The proof would  also use Lemma \ref{lem:workhorse}, that $\PD_M$ is a
  direct summand of $\PD_{\db M}$.
\end{remark}

\begin{lemma}\label{lem:boundary_and_PD_vorbereteit}
  If $i\colon M\into L$ is as in Lemma \ref{lem-boundary}, then
  \begin{equation*}
    \boundary \tensor \iota_{L,\boundary L}(x) = (-1)^{\deg x}\iota_{M,\boundary
    M}(i^*x)\tensor\boundary\qquad\forall x \in K_G^*(L),
  \end{equation*}
{ where  $\partial\in
    KK(C_0(M^\circ),C_0(L^\circ))$ is  the boundary
  element of the exact
  sequence of $C_0(L^\circ)\into C_0(L^\circ\cup M^\circ)\onto
  C_0(M^\circ)$}  as in Lemma \ref{lem-boundary} (here we abbreviate
  $L^\circ=L\setminus \boundary L$).
\end{lemma}
\begin{proof}
  We first recall a KK-description of $\iota_{L,\boundary L}$. It is given by the
  composition
  {\begin{equation*}
    KK(\{*\},L)\xrightarrow{\tau_{C_0(L^\circ)}}
    KK(L^\circ,L\times L^\circ)
    \xrightarrow{\tensor\mu} KK(L^\circ,L^\circ)
  \end{equation*}}
  where $\mu$ is the multiplication homomorphism.
% Now observe that
%  in  $KK(M^\circ,L^\circ)$.% the boundary
%   element of the exact
%   sequence of $C_0(L^\circ)\into C_0(L^\circ\cup M^\circ)\onto
%   C_0(M^\circ)$ as in Lemma \ref{lem-boundary}
 By graded commutativity  of the exterior Kasparov product, we
  therefore get that $\partial\tensor \iota_{L,\boundary L}$ equals the composition
  \begin{equation}\label{eq:bla}
     KK(\{*\},L)\xrightarrow{\tau_{M^\circ}}
       KK(M^\circ,L\times M^\circ)
       \xrightarrow{{(-1)^{\deg}\cdot}\tensor\partial}
     KK(M^\circ,L\times L^\circ) \xrightarrow{\mu}
     KK(M^\circ,L^\circ).
   \end{equation}

  Now observe that we have commutative diagrams of short exact sequences
  \begin{equation*}
    \begin{CD}
    C_0(L\times L^\circ) @>>> C_0(L\times ({L^\circ}){\cup M^\circ}) @>>>
    C_0(L\times M^\circ)\\
     @VVV @VV{=}V @V{i^*\times \id_{M^\circ}}VV\\
      C_0(L\setminus M\times  M^\circ \cup L\times L^\circ)
      @>>> C_0(L\times (L^\circ\cup M^\circ)) @>>> C_0(M\times M^\circ)\\
       @VV{\mu}V @VV{\mu}V @V{\mu}VV\\
       C_0(L^\circ) @>>> C_0(L^\circ\cup M^\circ) @>>> C_0(M^\circ)
  \end{CD}
\end{equation*}
Using naturality of the boundary map, we observe that the composition of the
last two arrows of \eqref{eq:bla} coincides with the composition
{\begin{equation*}
  KK(M^\circ,L\times M^\circ) \xrightarrow{i^*}
  KK(M^\circ,M\times  M^\circ )
  \xrightarrow{\tensor\mu}
  KK(M^\circ,M^\circ) \xrightarrow{\tensor\partial}
  KK(M^\circ,L^\circ). 
\end{equation*}}
As $i^*$ commutes with the exterior product with $C_0(M^\circ)$, this implies
the assertion.
\end{proof}

\subsection{Proof of Lemma \ref{lem-boundary}}
\label{sec:appendixprooflemma}

We finish by proving Lemma \ref{lem-boundary}. Recall that
it states
\begin{lemma}
  Let $L$ be a \spg manifold with boundary $\partial L$, let $M$ be a
  $G$-invariant submanifold of $\partial L$ with boundary $\partial M$ such that
  $\dim M=\dim L-1$ and let   $\partial\in
KK^G_1(C_0( M\setminus\partial M),C_0(L\setminus\partial L))$ be the boundary element
associated to the exact sequence
$$0\to C_0(L\setminus\partial L)\to C_0( (L\setminus \partial L)
 \cup (M\setminus\partial M))\to C_0(M\setminus \partial M)\to 0.$$
Then $[\partial]\otimes [L\setminus \partial L]=[M\setminus \partial M]$.
\end{lemma}

\begin{proof}
Using a $G$-invariant metric on $L$ and a corresponding collar, $(0,1]\times
(M\setminus \partial M)$
can be viewed as a $G$-invariant open  neighborhood of   $\{1\}\times
(M\setminus\partial M)$ in  $(L\setminus\partial L)
\cup (M\setminus\partial M)$. Moreover,  the  inclusion $C_0((0,1]\times
(M\setminus \partial M))\hookrightarrow
C_0(( L
\setminus \partial L)\cup (M\setminus\partial M))$  gives rise to the
following commutative diagram with exact rows \green{(write
$L^\circ:=L\setminus\boundary L$, $M^\circ:=M\setminus\boundary M$)
{\small\begin{equation*}
   \begin{CD}
    0  @>>>  C_0(L^\circ) @>>>  C_0( (L^\circ)
\cup (M^\circ)) @>>>  C_0(M^\circ)  @>>> 0 \\
 & &  @AAA @AAA @AAA\\
  0  @>>> C_0((0,1)\times (M^\circ))@>>> C_0((0,1]\times
  (M^\circ)) @>>> C_0( M^\circ)  @>>>  0 .
  \end{CD}
 \end{equation*}
}}
By naturality of the boundary homomorphism and since by \cite[Proposition
11.2.12]{MR1817560} (for the non-equivariant case, but the equivariant one
follows along identical lines) the restriction of
\green{$[L^\circ]$ to $(0,1)\times (M^\circ)$ is
$[(0,1)\times (M^\circ)]$, the statement of the lemma amounts to show that  $$[\partial']\otimes [(0,1)\times  (M^\circ)]=[ M^\circ],$$ where $\partial'\in
KK^G_1(C_0( M^\circ),C_0((0,1)\times (M^\circ)))$ is the boundary element
associated to the bottom exact sequence of the diagram above.
Viewing $ M^\circ$ as an invariant open subset of
$\db M$, using naturality of boundaries in the following commutative diagram
with exact rows
{\small \begin{equation*}
   \begin{CD}
0  @>>> C_0((0,1)\times \db M )@>>> C_0((0,1]\times \db M)
@>>> C_0(\db  M)  @>>>  0 \\
 & &  @AAA @AAA @AAA\\
  0  @>>> C_0((0,1)\times (M^\circ))@>>> C_0((0,1]\times
  (M^\circ)) @>>> C_0( M^\circ)  @>>>  0 ,
  \end{CD}
 \end{equation*}
}
and since the elements $[M^\circ]$ of
$KK^{G}_{*}(C_0(M^\circ),\C)$ and $[(0,1)\times
(M^\circ)]$ of $KK^{G}_{*}(C_0((0,1)\times (M^\circ)),\C)$ are
 the restrictions of $[\db M]$ to $M^\circ$ 
and of $[(0,1)\times \db
 M]$ to $(0,1)\times
(M^\circ)$, respectively, we can indeed assume without loss of
generality that $M$ has no boundary.}

Observe now that in the exact sequence in (non-equivariant) K-homology 
\begin{equation*}
  0\to C_0((0,1))\to C_0((0,1])\to C(\{1\})\to 0
\end{equation*}
by the well known principle that ``the boundary of the Dirac element is
the Dirac element of the boundary'' we indeed
observe $[\partial'']\tensor [(0,1)] =[\{1\}]$ in $KK_0(C(\{1\}),\C)$, compare
\cite[Propositions 9.6.7, 11.2.15]{MR1817560}. We can
now take the exterior Kasparov product of everything with
$[M]\in KK^G_{\dim M}(C_0(M),\complexs)$. By naturality of this Kasparov
product, we obtain $[\partial']\tensor [(0,1)]\tensor [M] = [\{1\}]\tensor
[M]$. Finally, we know that the fundamental class of a product is the exterior
Kasparov product of the fundamental classes, compare again \cite[Proposition
11.2.13]{MR1817560}; 
the equivariant situation follows similarly. This implies the desired relation
$[\partial']\tensor [(0,1)\times M]=[M] \in KK_{\dim M}(C_0(M),\C)$.

%% But in this case, if ${S}^1$ stands for the unit circle, then  $[(0,1)\times
%% M]$ is also the restriction of $[{S}^1\times
%% M]$ to $(0,1)\times
%% M$ under any orientation preserving embedding of $(0,1)$ into
%% $S^1$. Therefore, if $\theta\colon C_0((0,1)\times
%% M)\hookrightarrow C({S}_1\times
%% M)$ denotes the inclusion map and $[\theta]$ the induced

\end{proof}

\end{appendix}

%\nocite{ProcISFA,ProcSPM38}
{\small
\bibliographystyle{plain}
\bibliography{K_homology}
}

\end{document}